\documentclass[12pt]{article}

\usepackage{pifont}
\usepackage{amssymb}
\usepackage{amsmath}
\usepackage{fancyhdr}
\usepackage{setspace}
\usepackage{enumitem}
\setlist{nolistsep}

\renewenvironment{thebibliography}[1]{
  \begin{oldthebibliography}{#1}
    \setlength{\itemsep}{0em}
    \setlength{\parskip}{0em}
}
{
  \end{oldthebibliography}
}
\usepackage{graphicx}
\usepackage[usenames,dvipsnames]{color}
\usepackage{epsfig}
\usepackage[mathscr]{eucal}
\newtheorem{theorem}{Theorem}[section]
\newtheorem{lemma}[theorem]{Lemma}
\newtheorem{corollary}[theorem]{Corollary}

\newtheorem{definition}[theorem]{Definition}

\newtheorem{remark}[theorem]{Remark}

\newtheorem{result}[theorem]{Result}
\newenvironment{proof}{\noindent{\bf Proof}\hspace{0.5em}}
    { \null  \hfill $\square$ \par}

\newcommand\B{{\mathscr B}}
\newcommand\U{{\cal U}}
\newcommand{\X}{\mathcal X}

\newcommand{\R}{\mathcal R}
\newcommand\C{{\cal C}}

\newcommand\D{{\cal D}}
\newcommand\V{{\cal V}}

\newcommand\N{{\cal N}}

\renewcommand\L{{\mathscr L}}
\renewcommand{\S}{\mathcal S}
\renewcommand{\P}{\mathcal P}
\newcommand{\K}{\mathcal K}

\newcommand{\Q}{\mathscr Q}

\newcommand{\si}{\Sigma_\infty}
\newcommand{\abb}{{\cal A(\cal S)}}
\newcommand{\pbb}{{\cal P(\cal S)}}
\newcommand{\li}{\ell_\infty}

\newcommand{\takeaway}{\backslash}
\renewcommand\setminus{\backslash}
\newcommand{\st}{\,|\,}

\newcommand\PGL{{\rm PGL}}
\newcommand\GF{{\rm GF}}
\newcommand\PG{{\rm PG}}

\newcommand{\Label}{\label}
 
    \newcommand{\Fq}{\mathbb F_{q}}
\newcommand{\Fqq}{\mathbb F_{q^2}}
\newcommand{\plus}{\raisebox{.2\height}{\scalebox{.5}{+}}}
\newcommand{\Cplus}{\C^{{\rm \pmb\plus}}}

\newcommand{\Pbar}{{P}}
\newcommand{\Qbar}{{Q}}
\newcommand{\Pqbar}{P^q}
\newcommand{\Qqbar}{Q^q}
\newcommand{\Tbar}{{T}}
\newcommand{\Tqbar}{T^q}
\newcommand{\Konestar}{{\K}_1 ^{\mbox{\tiny\ding{73}}}}
\newcommand{\Ksstar}{{\K}_s^{\mbox{\tiny\ding{73}}}}

\newcommand{\Conestar}{{\C}_1 ^{\mbox{\tiny\ding{73}}}}
\newcommand{\Cqtwostar}{{\C}_{q^2}^{\mbox{\tiny\ding{73}}}}

\newcommand{\gstar}{g^{\mbox{\tiny\ding{72}}}}
\newcommand{\gqstar}{{g^q}^{\mbox{\tiny\ding{72}}}}
\newcommand{\Cstarstar}{[\C]^{\mbox{\tiny\ding{72}}}}
\newcommand{\sistar}{{\Sigma}_\infty^{\mbox{\tiny\ding{73}}}}
\newcommand{\sistarstar}{{\Sigma}_\infty^{\mbox{\tiny\ding{72}}}}
\newcommand{\Bstarstar}{[\B]^{\mbox{\tiny\ding{72}}}}
\newcommand{\Tstarstar}{[T]^{\mbox{\tiny\ding{72}}}}
\newcommand{\Cplusstarstar}{[\Cplus]^{\mbox{\tiny\ding{72}}}}
\newcommand{\PiCstar}{{\PiC}^{\mbox{\tiny\ding{73}}}}
\newcommand{\Cstar}{{[\C]^{\mbox{\tiny\ding{73}}}}}
\newcommand{\Qonestar}{{\Q}_\infty^{\mbox{\tiny\ding{73}}}}
\newcommand{\Qtwostar}{{\Q}_0^{\mbox{\tiny\ding{73}}}}
\newcommand{\Qtstar}{{\Q}_t ^{\mbox{\tiny\ding{73}}}}
\newcommand{\Qonestarstar}{{\Q}_\infty^{\mbox{\tiny\ding{72}}}}
\newcommand{\Qtwostarstar}{{\Q}_0^{\mbox{\tiny\ding{72}}}}

\newcommand{\newQonestar}{{\Q}^{\mbox{\tiny\ding{73}}}_1}

\newcommand{\newQtwostar}{{\Q}^{\mbox{\tiny\ding{73}}}_2}
\newcommand{\newQzerostar}{{\Q}^{\mbox{\tiny\ding{73}}}_0}

\newcommand{\Tstar}{{[T]^{\mbox{\tiny\ding{73}}}}}

\newcommand{\Lstar}{{[L]^{\mbox{\tiny\ding{73}}}}}
\newcommand{\bstar}{{[b]^{\mbox{\tiny\ding{73}}}}}

\renewcommand{\star}{{^{\mbox{\tiny\ding{73}}}}}
\newcommand{\blackstar}{{^{\mbox{\tiny\ding{72}}}}}

\newcommand{\alphastar}{{\alpha{\mbox{\raisebox{.9\height}{\scalebox{1}{{\tiny\ding{73}}}}}}}}

\newcommand{\Bstar}{{[\B]^{\mbox{\tiny\ding{73}}}}}
\newcommand{\Cplusstar}{[\Cplus]^{\mbox{\tiny\ding{73}}}}
\newcommand{\Pstar}{{[P]^{\mbox{\tiny\ding{73}}}}}
\newcommand{\Qptstar}{{[Q]^{\mbox{\tiny\ding{73}}}}}
\newcommand{\Pstarstar}{{[P]^{\mbox{\tiny\ding{72}}}}}
\newcommand{\Qptstarstar}{{[Q]^{\mbox{\tiny\ding{72}}}}}

\newcommand{\Kstarstar}{{\K^{\mbox{\tiny\ding{72}}}}}
\newcommand{\Nstar}{{\N^{\mbox{\tiny\ding{73}}}}}
\newcommand{\Nonestar}{{\N}_1^{\mbox{\tiny\ding{73}}}}

\newcommand{\Nstarstar}{{\N^{\mbox{\tiny\ding{72}}}}}
\newcommand{\Vstar}{{\V^{\mbox{\tiny\ding{73}}}}}

\newcommand{\Qstar}{{\Q^{\mbox{\tiny\ding{73}}}}}

\newcommand{\mstar}{{m^{\mbox{\tiny\ding{73}}}}}

\newcommand{\mathscrOstar}{{[{\mathscr O}]^{\mbox{\tiny\ding{73}}}}}

\newcommand{\mathscrOstarstar}{{[{\mathscr O}]^{\mbox{\tiny\ding{72}}}}}

\newcommand{\gPQ}{PQ^q}
\newcommand{\gQP}{P^qQ}

\newcommand{\gAB}{AB^q}
\newcommand{\mP}{m_{\mbox{\raisebox{-.25\height}{\scalebox{.55}{{\boldmath{$ \!P$}}}}}}}

\newcommand{\PiN}{\Pi_{\mbox{\raisebox{-.1\height}{\scalebox{.65}{{\boldmath{$\N$}}}}}}}
\newcommand{\PiC}{\Pi_{\mbox{\raisebox{-.1\height}{\scalebox{.65}{{\boldmath{$\C$}}}}}}}

\newcommand{\gPiQi}{P_iQ_i^q}
\newcommand{\elllP}{\ell_{\mbox{\raisebox{-.15\height}{\scalebox{.55}{{{$\!P$}}}}}}}

\newcommand{\ellP}{PP^q}

\begin{document}

\title{Conics in Baer subplanes}

\author{S.G. Barwick, Wen-Ai Jackson and Peter Wild}
\date{}
%
%
%
\maketitle

{\small

}


\begin{abstract} This article studies conics and subconics of $\PG(2,q^2)$ and their representation in the Andr\'e/Bruck-Bose 
setting in $\PG(4,q)$. In particular, we investigate their relationship with the transversal lines of the regular spread. 
 The main result is to show that a conic in a tangent Baer subplane of $\PG(2,q^2)$  corresponds in $\PG(4,q)$ to a normal rational curve that meets the transversal lines of the regular spread. Conversely, every 3 and 4-dimensional normal rational curve in $\PG(4,q)$ that meets the transversal lines of the regular spread  corresponds to  a conic in a tangent Baer subplane of $\PG(2,q^2)$. 
\end{abstract}

\section{Introduction}

This article investigates the representation of conics and subconics of $\PG(2,q^2)$ in the 
Bruck-Bose representation in $\PG(4,q)$. The Bruck-Bose representation of $\PG(2,q^2)$  uses a regular spread $\S$ in the hyperplane at infinity of $\PG(4,q)$. The regular spread $\S$ has two unique transversal lines $g,g^q$ in the quadratic extension $\PG(4,q^2)$. There are several known characterisations of objects of $\PG(4,q)$ in terms of their relationship with these transversal lines. 
Firstly, a conic $\C$ in $\PG(4,q)$ corresponds to a Baer subline of $\PG(2,q^2)$ iff  the extension of $\C$ to a conic of $\PG(4,q^2)$ contains a point of $g$ and a point of $g^q$, see   \cite{CasseQuinn2002}.  A ruled cubic surface $\V$ in $\PG(4,q)$ corresponds to a Baer subplane of $\PG(2,q^2)$ iff the extension of $\V$ to $\PG(4,q^2)$ contains $g$ and $g^q$, see  \cite{CasseQuinn2002}. Further, an orthogonal cone $\U$ corresponds to a classical unital of $\PG(2,q^2)$ iff the extension of $\U$ to $\PG(4,q^2)$ contains $g$ and $g^q$, see   \cite{mets97}. Hence the 
interaction  of certain objects with the transversals of $\S$  is intrinsic to their characterisation in $\PG(2,q^2)$.
In this article we study conics and subconics of $\PG(2,q^2)$, and determine their relationship with the transversals of $\S$ in the Bruck-Bose setting in  $\PG(4,q)$.  In particular, we characterise normal rational curves of $\PG(4,q)$ whose extension meets the transversals as subconics of $\PG(2,q^2)$. 

The article is set out as follows. Section~\ref{sec:background} introduces the necessary background and proves  some preliminary results. In particular, in order to study how objects of the Bruck-Bose representation relate to the transversals of the regular spread $\S$, we formally define in Section~\ref{sec:def-special} the notion of special sets in $\PG(4,q)$.  Further, in Section~\ref{sec:partition}, we consider a Baer subplane $\B$ tangent to $\li$, and give a geometric construction via $\PG(4,q)$ that partitions the affine points of $\B$ into $q$ conics, one of which is degenerate.  
 In Section~\ref{sec:special-other-Baer}, we discuss how the notion of specialness relates to the known Bruck-Bose representation of Baer sublines and Baer subplanes. 
 
 In Section~\ref{sec:adult-conic}, we investigate non-degenerate  conics of $\PG(2,q^2)$ in the $\PG(4,q)$  Bruck-Bose representation. 
In particular, we investigate the corresponding structure in the quadratic extension to $\PG(4,q^2)$. We  show that in $\PG(4,q^2)$, the  (extended) structure corresponding to a non-degenerate conic $\mathcal O$ is the intersection of two quadrics which meet $g$ in the two points (possibly repeated or in an extension) corresponding to $\mathscr O\cap\li$.

 In Section~\ref{sec:subconics} we characterise the Bruck-Bose representation of  conics contained in Baer subplanes. In $\PG(2,q^2)$, let $\B$ be a Baer subplane  tangent to $\li$,  and $\C$  a non-degenerate conic contained in $\B$. We show that in $\PG(4,q)$, $\C$ corresponds to a normal rational curve that meets the transversals of the regular spread. Conversely, we characterise every normal rational curve in $\PG(4,q)$  that meets the transversals of the regular spread as corresponding to a non-degenerate  conic in a Baer subplane of $\PG(2,q^2)$. 
 
While the proofs in Section~\ref{sec:adult-conic} are largely coordinate based, the proofs in Section~\ref{sec:subconics} use geometrical arguments.

\section{Background and Preliminary Results}\Label{sec:background}

In this section we give the necessary background, introduce the notation we use in this article, and prove a number of preliminary results.

\subsection{Conjugate points}

For $q$ a prime power, we denote the unique finite field of order $q$ by $\mathbb F_q$.
We use the phrase conjugate points in several different settings. Firstly, consider the automorphism $x\mapsto x^q$ for $x\in{{\mathbb F}_{q^r}}$, it induces an automorphic collineation of $\PG(n,q^r)$
where a point $X=(x_0,\ldots,x_n)\mapsto X^q =(x_0^q,\ldots,x_n^q)$. The points $X,X^q,\ldots,X^{q^n-1}$ are called \emph{conjugate}. Secondly, let $\B$ be a Baer subplane of $\PG(2,q^2)$, them there is a unique involutory collineation that fixes $\B$ pointwise, and we call this map \emph{conjugacy with respect to $\B$.}  
Note that  $P,Q\in\li$ are conjugate with respect to the secant Baer subplane $\B$ if and only if $P,Q$ are conjugate with respect to the Baer subline $\B\cap\li$.

\subsection{Spreads in $\PG(3,q)$}\Label{sec:spreads}


The following construction of a regular spread of $\PG(3,q)$ will be
needed, see \cite{hirs91} for
more information on spreads. Embed $\PG(3,q)$ in $\PG(3,q^2)$ and let $g$ be a line of
$\PG(3,q^2)$ disjoint from $\PG(3,q)$. 
The 
 line $g$ has a conjugate line $g^q$ with respect to the map $x\mapsto x^q$, $x\in{{\mathbb F}_{q^2}}$, and $g^q$ is also  disjoint from $\PG(3,q)$. Let $P_i$ be
a point on $g$; then the line $\langle P_i,P_i^q\rangle$ meets
$\PG(3,q)$ in a line. As $P_i$ ranges over all the points of  $g$, we obtain 
$q^2+1$ lines of $\PG(3,q)$ that partition $\PG(3,q)$. These lines form a
regular spread $\S$ of $\PG(3,q)$. The lines $g$, $g^q$ are called the (conjugate
skew) {\em transversal lines} of the regular spread $\S$. Conversely, given a regular spread $\S$ 
in $\PG(3,q)$,
there is a unique pair of  transversal lines in $\PG(3,q^2)$ that generate
$\S$ in this way.

\subsection[Bruck-Bose]{The Bruck-Bose representation}\Label{BBintro}

We will use the linear representation of a finite
translation plane  of dimension at most two over its kernel,
due independently to
Andr\'{e}~\cite{andr54} and Bruck and Bose
\cite{bruc64,bruc66}. 
Let $\si$ be a hyperplane of $\PG(4,q)$ and let $\S$ be a spread
of $\si$. We use the phrase {\em a subspace of $\PG(4,q)\takeaway\si$} to
  mean a subspace of $\PG(4,q)$ that is not contained in $\si$.  Consider the following incidence
structure:
the {\sl points} of $\abb$ are the points of $\PG(4,q)\takeaway\si$; the {\sl lines} of $\abb$ are the planes of $\PG(4,q)\takeaway\si$ that contain
  an element of $\S$; and {\sl incidence} in $\abb$ is induced by incidence in
  $\PG(4,q)$.
Then the incidence structure $\abb$ is an affine plane of order $q^2$. We
can complete $\abb$ to a projective plane $\pbb$; the points on the line at
infinity $\li$ have a natural correspondence to the elements of the spread $\S$. We call this the {\em Bruck-Bose representation} of $\P(\S)$ in $\PG(4,q)$. 
The projective plane $\pbb$ is the Desarguesian plane $\PG(2,q^2)$ if and
only if $\S$ is a regular spread of $\si\cong\PG(3,q)$ (see \cite{bruc69}). 
We use the following notation in the Bruck-Bose setting. 
\begin{itemize}
\item  $\S$ is  a regular spread with transversal lines $g,g^q$.
\item An affine point of $\PG(2,q^2)\setminus\li$ is denoted with a capital letter, $A$ say, and 
 $[A]$ denotes the corresponding point of $\PG(4,q)\setminus\si$. 
 \item A point on $\li$ in $\PG(2,q^2)$ is denoted with an  over-lined capital letter, $\bar T$ say, and the corresponding spread line is denoted $[T]$. 
 \item The points of $\li$ are in 1-1 correspondence with the points of $g$; for a point $\bar T\in\li$, we denote the corresponding point of $g$ by $T$.
\item A set of points $\X$ in $\PG(2,q^2)$ corresponds to a set of points denoted {$[\X]$} in $\PG(4,q)$. 
\end{itemize}

We will work in the extension of $\PG(4,q)$ to $\PG(4,q^2)$ and to $\PG(4,q^4)$. 
 Let $\K$ be a 
 primal of $\PG(4,q)$, 
 so $\K$ is the set of points of $\PG(4,q)$ satisfying a homogeneous equation $f(x_0,\ldots,x_4)=0$, with coefficients in $\Fq$. 
We define $\K\star$ to be the  (unique) primal   of $\PG(4,q^2)$  which is the set of points of $\PG(4,q^2)$  satisfy the same homogeneous equation $f=0$. 
Note that if $\K=\Pi$ is an $r$-dimensional subspace of $\PG(4,q)$, then $\Pi\star$ is the (unique) $r$-dimensional subspace of $\PG(4,q^2)$ containing $\Pi$. 
Further, if $\V$ is a variety of $\PG(4,q)$, so $\V$ is the intersection of primals $\K_1,\ldots,\K_s$, then we define $\V\star=\Konestar\cap\cdots\cap\Ksstar$.
Similarly, we can extend a primal $\K$ to $\PG(4,q^4)$, and we denote the resulting set by $\Kstarstar$. 
The transversals $g,g^q$ of the regular spread  $\S$ lie in $\PG(4,q^2)$, and we denote their extensions to lines of  $\PG(4,q^4)$ by $\gstar,\gqstar$ respectively.

\subsection{Ruled cubic surfaces in $\PG(4,q)$}

A ruled cubic surface $\V$ of $\PG(4,q)$ consists of a line directrix $t$, a conic directrix $\C$ lying in a plane disjoint from $t$,  and a set of $q+1$ pairwise disjoint generator lines joining the points of $t$ and $\C$ according to a projectivity $\omega\in\PGL(2,q)$. That is, let  $\theta,\phi\in\mathbb F_{q}\cup\{\infty\}$ be the non-homogeneous coordinates of $t,\C$ respectively,  and $\omega\colon(1, \theta)\mapsto(1,\phi)$, then the generators of $\V$ are the lines joining points of $t$ to the corresponding point of $\C$ under $\omega$. 
We will need the following result which shows how hyperplanes of $\PG(4,q)$ meet a ruled cubic surface.

\begin{result}\Label{3-space-meets-ruled} \cite{quinn-conic} A hyperplane of $\PG(4,q)$ meets a ruled cubic surface in  one of the following. 
\begin{itemize}
\item The line directrix; $(q^2-q)/2$ hyperplanes do this.
\item  The line directrix and one generator line; $q+1$ hyperplanes do this.
\item  The line directrix and two generator lines; $(q^2+q)/2$ hyperplanes do this.
\item A conic and a generator line; $q^3+q^2$ hyperplanes do this.
\item A twisted cubic curve (which meets the line directrix in a unique point); $q^4-q^2$ hyperplanes do this.
\end{itemize}
\end{result}

\begin{corollary}\Label{lem:tc-brs}
Let $\Pi$ be a hyperplane of $\PG(4,q)$ that meets a ruled cubic surface $\V$ in a twisted cubic $\N$. Then $\N$ meets each generator line of $\V$ in a unique point.
\end{corollary}

\begin{proof}
If $\N$ meets a generator line $\ell$ of $\V$ in two points, then the 3-space $\Pi$ containing $\N$ also contains $\ell$, which is not possible by  Result~\ref{3-space-meets-ruled}. Hence $\N$ meets each generator line in at most one point. As $\N$ contains $q+1$ points, each generator of $\V$ contains a unique point of $\N$. 
\end{proof}

There are two ways to extend the ruled cubic surface to $\PG(4,q^2)$, we show that they are equivalent. 
The ruled cubic surface $\V$ is a variety whose points are  the exact intersection of three quadrics, $\V=\Q_0\cap\Q_1\cap\Q_2$ (see for example \cite{FFA}). So extending this variety to  $\PG(4,q^2)$ yields $\Vstar=\newQzerostar\cap\newQonestar\cap\newQtwostar$. Alternatively, we can 
  extend $\V$ to $\PG(4,q^2)$ as in \cite{CasseQuinn2002}: namely extending the line directrix $t$ and conic directrix $\C$ to $\PG(4,q^2)$ by taking $\theta,\phi\in\mathbb F_{q^2}\cup\{\infty\}$, and extending the projectivity $\omega$ to act over $\mathbb F_{q^2}$. We denote this extension by $\V'$, thus $\V'$ is the ruled cubic surface with line directrix $t\star$, conic directrix $\C\star$, and ruled using the  (extended) projectivity $\omega$. 
We show that these two extensions $\Vstar$, $\V'$ are the same. The surface  $\V$ contains exactly $q^2$ conics $\C_1,\ldots,\C_{q^2}$, and these conics cover each point of $\V\setminus t$ $q$ times  (see \cite{UnitalBook} for more details). Hence both sets $\Vstar,\V'$ contain the extended conics $\Conestar,\ldots,\Cqtwostar$. Moreover,  these conics together with  $t\star$ cover all the points of $\V'$, and so $\Vstar$ contains $\V'$. However, $\Vstar$ is the intersection of three quadrics over $\mathbb F_{q^2}$, whose intersection over $\mathbb F_{q}$ is a ruled cubic surface. By \cite{BV}, all ruled cubic surfaces are projectively equivalent,  hence $\Vstar$ and $\V'$ are the same ruled cubic surface of $\PG(4,q^2)$.

\subsection{Coordinates in Bruck-Bose}\Label{sec:background-coord}

We now show how the coordinates of points in
$\PG(2,q^2)$ relate to the coordinates of the corresponding points
in the Bruck-Bose representation in $\PG(4,q)$. See \cite[Section 3.4]{UnitalBook} for more details. 
Let $\tau$ be a primitive element in ${{\mathbb F}_{q^2}}$ with primitive
polynomial
$
x^2-t_1x- t_0
$
 over ${{\mathbb F}_q}$. Then every
element $\alpha\in{{\mathbb F}_{q^2}}$ can be uniquely written as
$\alpha=a_0+a_1\tau$ with $a_0,a_1\in{{\mathbb F}_q}$. That is,
${{\mathbb F}_{q^2}}=\{x_0+x_1\tau\mid x_0,x_1\in{{\mathbb F}_q}\}.$ It is useful to keep in mind the relationships: $\tau\tau^q=-t_0$, $\tau+\tau^q=t_1$, $t_0/\tau=-\tau^q=\tau-t_1$ and $\tau^{q^2}=1$.
Points in $\PG(2,q^2)$ have homogeneous coordinates
$(x,y,z)$ with $x,y,z\in{{\mathbb F}_{q^2}}$, not all zero.
We let the line at infinity $\li$ 
 have equation $z=0$, so affine points of
$\PG(2,q^2)$ have 
coordinates $(x,y,1)$, with $x,y\in{{\mathbb F}_{q^2}}$. 
Points in $\PG(4,q)$ have homogeneous coordinates
$(x_0,x_1,y_0,y_1,z)$ with
$x_0,x_1,y_0,y_1,z\in{{\mathbb F}_q}$, not all zero. We let the hyperplane at infinity
$\si$ have equation $z=0$, so the affine points of
$\PG(4,q)$ have coordinates
$(x_0,x_1,y_0,y_1,1)$, with $x_0,x_1,y_0,y_1 \in {{\mathbb F}_q}$. 
Let $A$ be an affine point  in $\PG(2,q^2)$ with coordinates $A=(x_0+x_1\tau,y_0+y_1\tau,z)$,  where $
x_0, x_1, y_0, y_1, z\in{{\mathbb F}_q}$, $z\neq 0$.
The map
\begin{eqnarray*}
\varphi: \PG(2,q^2)\setminus\li &\longrightarrow &\PG(4,q)\setminus\si
\\ \mbox {such that } \quad\quad A=(x_0+x_1\tau,y_0+y_1\tau,z)
&\longmapsto&[A]=(x_0,x_1,y_0,y_1,z).
\end{eqnarray*}
 is a bijection from the affine points of
$\PG(2,q^2)$ to the affine points of $\PG(4,q)$, called
the {\em Bruck-Bose map}.
We can extend this to a projective map; for a point $\bar T= (\delta,1,0)\in\li$, write 
$\delta=d_0+d_1 \tau\in{{\mathbb F}_{q^2}}$, $d_0,d_1\in{{\mathbb F}_q}$, then 
$$\bar T= (\delta,1,0)\ \longmapsto \ [T]= \big\langle\,
(d_0,d_1,1,0,0), \ (t_0 d_1,d_0+t_1d_1,0,1,0)\, \big\rangle.$$  The transversal lines $g,g^q$ of  $\S$ have  coordinates given by:
\begin{eqnarray*}
g&=&\big\langle \, A_0=(\tau^q,-1,0,0,0),\ A_1=(0,0,\tau^q,-1,0)\, \big\rangle,\\
g^q&=&\big\langle \, A_0^q=(\tau,-1,0,0,0),\ A_1^q=(0,0,\tau,-1,0)\, \big\rangle.\end{eqnarray*}
Recall that each line of the regular spread $\S$ meets the transversal $g$ of $\S$ in a point. 
The   1-1  correspondence between points of $\li$ and points of $g$ is:
$$ \bar T=   (\delta,1,0)\in\li\  \ \longleftrightarrow \ \ \Tbar=\delta A_0+A_1\in g,\quad\quad \delta\in{{\mathbb F}_{q^2}}\cup\{\infty\},$$
that is, $\Tbar=\Tstar\cap g$ and  $\Tstar=\Tbar\Tqbar$.

\subsubsection{Coordinates and the quartic extension $\PG(4,q^4)$}\Label{sec:notn-bar}

We will be interested in non-degenerate conics of $\PG(2,q^2)$, and one of the cases to consider is when a conic $\mathscr C$ is exterior to $\li$, and so meets $\li$ in two points which lie in the quadratic extension of $\PG(2,q^2)$ to $\PG(2,q^4)$.
That is, $\mathscr C$ meets $\li$ in two points $\bar P,\, \bar Q$ over ${{\mathbb F}_{q^4}}$. Note that $\bar P,\, \bar Q$ are conjugate with respect to the  map $x\mapsto x^{q^2}$, $x\in{{\mathbb F}_{q^4}}$, that is $\bar Q=\bar P^{q^2}$. There is no direct representation for the point $\bar P$ in the Bruck-Bose representation in $\PG(4,q)$. However, there is a related point in the quartic  extension  $\PG(4,q^4)$. We can extend the  1-1  correspondence between  points $\li$ and points of $g$ to a  1-1  correspondence between points of the quadratic extension of $\li$ and points  of  the extended transversal $\gstar$ in $\PG(4,q^4)$, so
$$  \ \ \bar P=(\alpha,1,0)\ \ \ \  \longleftrightarrow \ \ \ \  \Pbar= \alpha A_0+A_1\in \gstar,\quad \alpha\in{{\mathbb F}_{q^4}}\cup\{\infty\}.$$
If $ \bar P=(\alpha,1,0)$ for some $\alpha\in{{\mathbb F}_{q^4}}\setminus{{\mathbb F}_{q^2}}$, that is $ \bar  P\in\PG(2,q^4)\setminus\PG(2,q^2)$, then in $\PG(4,q^4)$, the corresponding point  $\Pbar$ lies in $\gstar\setminus g$, and the conjugate point $\Pqbar=\alpha^q A_0^q+A_1^q$ lies on $\gqstar\setminus g^q$.  As $\bar P\notin\PG(2,q^2)$, the line $\ellP$ is  not a line of the spread $\S$; $\ellP$ is a line of $\PG(4,q^4)$ that does not meet $\si$.

\subsection{$g$-special sets}\Label{sec:def-special}

When studying curves of $\PG(2,q^2)$ in the $\PG(4,q)$ Bruck-Bose setting, the transversals $g,g^q$ of the regular spread $\S$ play an important role in characterisations. 
Let $\V$ be a variety or rational curve in $\PG(4,q)$, 
we are interested in how $\Vstar$ meets $g,g^q$  
 in the extension to $\PG(4,q^2)$. Note that if $\Vstar$ meets $g$ in a point $P$, then as $\V$ is defined over ${{\mathbb F}_q}$, $\Vstar$ also meets $g^q$ in the point $P^q$. 
 A non-degenerate conic $\C$ in $\PG(4,q)$ is called a {\em $g$-special conic} if in $\PG(4,q^2)$, $\C\star$ contains one point of  $g$, and one point of $g^q$.
 A twisted cubic $\N$ in $\PG(4,q)$ is  called a {\em $g$-special twisted cubic} if in $\PG(4,q^2)$, $\Nstar$ contains one point of  $g$, and one point of $g^q$.
 A 4-dimensional normal rational curve $\N$ in $\PG(4,q)$ is  called a {\em $g$-special normal rational curve} if in $\PG(4,q^2)$, $\Nstar$ contains two  points of  $g$ (possibly repeated) and two points of $g^q$. Further, $\N$ is called  {\em $\gstar$-special}  if in the {quartic} extension $\PG(4,q^4)$, $\Nstarstar$ contains two  points of the  extended transversal $\gstar\setminus g$.
 A ruled cubic surface $\V$ in $\PG(4,q)$ is  called a {\em $g$-special ruled cubic surface} if in $\PG(4,q^2)$, $\Vstar$ contains   $g$ and $g^q$.

\subsection{Representations of Baer sublines and subplanes}
 
We use the following representations of Baer sublines and subplanes  of $\PG(2,q^2)$ in $\PG(4,q)$, see \cite{UnitalBook} for more details.  

\begin{result}\Label{BB-Baer} Let $\S$ be a regular spread in a $3$-space $\si$ in $\PG(4,q)$ and consider the representation of the Desarguesian  plane $\P(\S)=\PG(2,q^2)$ defined by the Bruck-Bose construction. 
\begin{enumerate}
\item A Baer subline of $\li$ in $\PG(2,q^2)$ corresponds to a regulus of  $\S$. 
\item A  Baer subline of $\PG(2,q^2)$ that meets $\li$ in a point corresponds to a line of $\PG(4,q)\setminus\si$. 
\item A Baer subplane of $\PG(2,q^2)$  secant to $\li$ corresponds  to a plane of $\PG(4,q)\setminus\si$ not containing a spread line.
\item A   Baer subline of $\PG(2,q^2)$ that is disjoint from $\li$  corresponds  in $\PG(4,q)$ to a $g$-special conic.
\item A Baer subplane tangent to $\li$ at a point $\bar T$ corresponds  in $\PG(4,q)$ to a $g$-special ruled cubic surface containing the corresponding spread line $[T]$.
\end{enumerate}
Moreover, the converse of each of these correspondences holds.
\end{result}

\begin{remark}\Label{remark:BB-Bose}{\rm The correspondences in parts 2 and 3 are not exact at infinity. The exact at infinity representation of a Baer subline that meets $\li$ in a point $T$ is an affine line that meets the spread line $[T]$ {\em union} with the spread line $[T]$. Similarly,  the exact at infinity representation of a secant Baer subplane is a plane $\alpha$ not containing a spread line, {\em union} the lines of $\S$ that $\alpha$ meets. 
}\end{remark}

\subsection{Representations of  subconics}

The representation of non-degenerate conics contained  in a Baer subplane was considered in \cite{quinn-conic}.

\begin{result}\Label{cath-conic} \cite{quinn-conic} 
Let $\C$ be a non-degenerate conic contained  in a Baer subplane $\B$ of $\PG(2,q^2)$.
\begin{enumerate}
\item  Suppose $\B$ is secant to $\li$, then $\C$ corresponds to a non-degenerate conic in the plane $[\B]$ of $\PG(4, q)$.  
\item  Suppose $\B$ is tangent to $\li$, $\B\cap\li\in\C$, and $q\geq 3$. Then $\C$ corresponds to a twisted cubic  on the  ruled cubic surface $[\B]$ of $\PG(4,q)$.
\item Suppose $\B$ is tangent to $\li$,  $\B\cap\li\notin\C$, and $q\geq 4$. Then $\C$ corresponds to a 4-dimensional normal rational curve on the ruled cubic surface $[\B]$ of $\PG(4,q)$.
\end{enumerate}
\end{result}

In Section~\ref{sec:subconics}, we show that the 3- and 4-dimensional normal rational curves of Result~\ref{cath-conic} are $g$-special. Conversely, we show that every $g$-special normal rational curve in $\PG(4,q)$ corresponds to a non-degenerate conic contained  in a tangent Baer subplane.

\begin{remark}\Label{subconic-exact}
{\rm
The correspondence in Result~\ref{cath-conic} parts 1 and 2  is not exact at infinity (compare with 
 Remark~\ref{remark:BB-Bose}). For example, in part 2, the point $\bar T=\B\cap\li$ is in $[\C]$, and the twisted cubic $[\C]$ meets $\si$ in a point of $[T]$. The exact-at-infinty representation is:
the set $[\C]$ is a twisted cubic {\em union} the spread line $[T]$.  We  use the simpler, not exact-at-infinity correspondence given in Result~\ref{cath-conic} as it does not lead to any confusion.
}
\end{remark}

\subsection{The circle geometry $CG(2,q)$}\Label{sec:hyp-cong}

Circle geometries $CG(d,q)$, $d\geq 2$ were introduced in \cite{BruckI,BruckII}, and we summarise the results we need here.  Note that $CG(2,q)$ is an inversive plane. We can construct $CG(2,q)$ from the line $\PG(1,q^2)$, in this case the circles  are the Baer sublines of $\PG(1,q^2)$. Equivalently, we can construct $CG(2,q)$ from the lines of a regular spread $\S$ of $\PG(3,q)$, in this case the circles  are the  reguli contained in  $\S$.
Using the representation of $CG(2,q)$ as $\li\cong\PG(1,q^2)$,  we can use properties of the circle geometry to deduce several  properties of the projective plane $\PG(2,q^2)$. 
If $\bar P,\, \bar Q$ are two distinct points on $\li$ in $\PG(2,q^2)$, 
then there is a unique partition of $\li$ into $\bar P,\, \bar Q$ and $q-1$ Baer sublines $\ell_1,\ldots,\ell_{q-1}$, where the points $\bar P,\, \bar Q$ are conjugate with respect to each Baer subline $\ell_i$. Further, if 
 $\B$ is a Baer subplane  secant to $\li$, such that $\bar P,\, \bar Q$ are conjugate with respect to $\B$, then $\B$ meets $\li$ in one of the Baer sublines $\ell_i$.
Of particular interest is an 
 application to conics.

\begin{result}\Label{part-sec-conic} Let $\mathscr O$ be a non-degenerate conic of $\PG(2,q^2)$ that meets $\li$ in $\{\bar P,\, \bar Q\}$. Then there is a unique partition of the $q^2-1$ affine points of $\mathscr O$ into $q-1$ subconics $\C_1,\ldots,\C_{q-1}$, lying in  Baer subplanes $\B_1,\ldots,\B_{q-1}$ which are secant to $\li$. Further, the    Baer sublines $\B_i\cap\li$ are either equal or disjoint
\end{result}

The properties of the circle geometry also lead to properties of  a regular spread $\S$ in $\PG(3,q)$. 
Let $g,g^q$ be the transversals of $\S$, so $g,g^q$ lie in $\PG(3,q^2)$. 
Consider the set of lines of $\PG(3,q^2)$ that meet both $g$ and $g^q$. This set is called 
the {\em hyperbolic congruence} of $g,g^q$ in \cite{Hirsch2}. Note that if two distinct lines in the hyperbolic congruence meet, then they meet on $g$ or $g^q$. The hyperbolic congruence contains  the extended spread lines $\Pstar=PP^q$ for $P\in g$ and the lines $PQ^q$ for distinct $P,Q\in g$.
 The lines $\gPQ$  have an interesting relationship with the regular spread $\S$.

\begin{result}\cite{BruckII}\Label{res:circle} Let $[P],[Q]$ be two lines of a regular spread $\S$ in $\PG(3,q)$, and denote their intersection with  transversal $g$ of $\S$ by  $P,Q$ respectively. Then there is a unique partition of $\S$ into $[P],[Q]$ and $q-1$ reguli $\R_1,\ldots,\R_{q-1}$. Denote the opposite regulus of $\R_i$ by $\R_i'$. Then the set $\{[P],[Q],\R_1',\ldots,\R_{q-1}'\}$ is a regular spread with   transversals 
 $\gPQ$,  $\gQP$.
   \end{result}

We will show that the lines in the hyperbolic congruence of $g,g^q$ are related to the Bruck-Bose representation of non-degenerate conics of $\PG(2,q^2)$ in $\PG(4,q)$.

\subsection{Normal rational curves contained in quadrics}

Next, we show that if a normal rational curve is contained in a quadric in  $\PG(4,q)$, then the containment also holds in the quadratic extension $\PG(4,q^2)$, provided $q$ is not small.

\begin{lemma}\Label{lem:nrc-extn}
 In $\PG(4,q)$, $q>7$, let $\N$ be a 4-dimensional normal rational curve and  $\Q$ a quadric, with $\N\subset \Q$. Then in the quadratic extension $\PG(4,q^2)$, $\Nstar\subset\Qstar$.
\end{lemma}

\begin{proof} Without loss of generality, let $\N=\{P_\theta=(1,\theta,\theta^2,\theta^3,\theta^4)\st\theta\in{{\mathbb F}_q}\cup\{\infty\}\}$.
Let $\Q$ have equation $g(x_0,x_1,x_2,x_3,x_4)=0$. Consider $g(P_\theta)=g(1,\theta,\theta^2,\theta^3,\theta^4)=f(\theta)$. As $\Q$ is a quadric, $f(\theta)$ is a polynomial in $\theta$ of degree at most 8. Now as $\N\subset\Q$, $f(P_\theta)=0$ for all $\theta\in{{\mathbb F}_q}\cup\{\infty\}$. So if $q+1>8$, $f$ is identically 0, and so
$f(P_\theta)=0$ for all $\theta\in{{\mathbb F}_{q^2}}$. Using  $\theta=\infty$, this implies that the coefficient of $\theta^8$ is zero, thus the degree of $f$ is at most $7$.  As $f(\theta)=0$ for the $q$ values of $\theta\in{{\mathbb F}_q}$, it follows that $f$ has $q$ roots, so if $q>7$ then $f$ is the zero polynomial, thus $f(\theta)=0$ for all $\theta$ in any extension of ${{\mathbb F}_q}$, and so $g(P_\theta)=0$ for all $\theta$ in any extension of ${{\mathbb F}_q}$. So if $q>7$, the point  $P_\theta$, $\theta\in{{\mathbb F}_{q^2}}\cup\{\infty\}$, lies on $\Qstar$, and so $\Nstar\subset\Qstar$.
\end{proof}

The bound on $q$ in Lemma~\ref{lem:nrc-extn} is tight  as shown by the following example. In $\PG(4,7)$, let
$\N$ be the normal rational curve $\N=\{P_\theta=(1,\theta,\theta^2,\theta^3,\theta^4)\st\theta\in\GF(7)\cup\{\infty\}\}$ and let $\Q$ be the quadric with equation $f(x_0,x_1,x_2,x_3,x_4)=-x_0x_1-x_3^2+x_2x_4+x_3x_4$. First note that $f(P_\theta)=\theta^7-\theta=0$ for all $\theta\in\GF(7)$. Further,   $P_\infty=(0,0,0,0,1)$, so $f(P_\infty)=0$. Hence $\N\subset\Q$ in $\PG(4,7)$. Now extend $\GF(7)$ to $\GF(7^2)$ using a primitive element  $\tau$. The point $P_\tau=(1,\tau,\tau^2,\tau^3,\tau^4)$ lies in the extended curve $\Nstar$. However $f(P_\tau)=\tau^7-\tau\ne 0$ as $\tau\not\in\GF(7)$, and so $P_\tau$ does not lie on the extended quadric $\Qstar$, that is $\Nstar\not\subset\Qstar$.

%

\subsection{Baer pencils and partitions of Baer subplanes}\Label{sec:partition}

In this section we investigate the representation in $\PG(2,q^2)$ of a $3$-space of $\PG(4,q)$. We use this to partition tangent Baer subplanes into conics.

\begin{definition}
A {\sl Baer pencil} $\K$ in $\PG(2,q^2)$ is the cone of $q+1$ lines joining a vertex point $P$ to a  Baer subline base $b$.  An {\sl $\li$-Baer pencil} $\K$ is a Baer pencil with vertex in $\li$, and base  $b$ meeting $\li$ in a  point.
 \end{definition}

Let $\K$ be a Baer pencil, then every line  of $\PG(2,q^2)$ not through the vertex of $\K$  meets $\K$ in a Baer subline. Also note that an 
 $\li$-Baer pencil $\K$ contains  $\li$, and a further $q^3$ affine points.
It is straightforward to characterise the $\li$-Baer pencils of $\PG(2,q^2)$ in $\PG(4,q)$.

\begin{lemma}\Label{lemma-3-Baer} Let $\Pi$ be a $3$-space in $\PG(4,q)$ distinct from $\si$. 
Then $\Pi$ corresponds in $\PG(2,q^2)$ to an $\li$-Baer pencil with vertex corresponding to the unique spread line in  $\Pi$. Conversely, any $\li$-Baer pencil in $\PG(2,q^2)$ corresponds to a $3$-space of $\PG(4,q)$.
\end{lemma}

We look at how $\li$-Baer pencils meet a tangent Baer subplane.

\begin{theorem}\Label{thm:partition-intro}
Let $\B$ be a  Baer subplane in $\PG(2,q^2)$ tangent to $\li$ at the point $\bar T=\B\cap\li$. An $\li$-Baer pencil with vertex $\bar P\neq\bar  T$ meets $\B$ in 
 either a non-degenerate conic through $\bar T$; or in 
 two lines, namely the unique line of $\B$ whose extension contains $\bar P$, and one line through $\bar T$.
Of the $\li$-Baer pencils with vertex $\bar P$, there are $q^2-1$ of the first type, and $q+1$ of the second type (each containing one of the $q+1$ lines of $\B$ through $\bar T$).
\end{theorem}

\begin{proof} In $\PG(4,q)$, let $X$ be a point on the spread line $[T]$, and let 
$\alpha=\langle X,[P]\rangle$. Label the $3$-spaces of 
$\PG(4,q)$ (not equal to $\si)$ that contain the plane $\alpha$ by $\L=\{\Pi_1,\ldots,\Pi_q\}$. By Lemma~\ref{lemma-3-Baer}, each  $3$-space in $\L$ 
corresponds to an $\li$-Baer pencil of $\PG(2,q^2)$ with vertex $P$. 
Result~\ref{3-space-meets-ruled} describes how a $3$-space meets the 
ruled cubic surface  $[\B]$. 
As each 3-space in $\L$ meets $[T]$ in one point, and the $3$-spaces in $\L$ partition the affine points, we deduce that 
one of the 3-spaces in $\L$, $\Pi_1$ say, meets
 $[\B]$ in  a conic and the generator line of $[\B]$ through $X$, and the remaining 3-spaces in $\L$ meet $[\B]$ in a twisted cubic $\N_i=\Pi_i\cap[\B]$, $i=2,\ldots,q$. 
 By Result~\ref{cath-conic}, the twisted cubics $\N_i$ each correspond in $\PG(2,q^2)$ to   non-degenerate conics in $\B$ that contains $\bar T$.
Note that there is a unique plane of $\PG(4,q)\setminus\si$ that contains the spread line $[P]$ and meets $[\B]$ in a conic; namely the plane that corresponds in $\PG(2,q^2)$ to the unique line $\mP$ through $\bar P$ that meets $\B$ in a Baer subline.
Hence $\Pi_1\cap[\B]$ contains the generator line $[m]$ of $[\B]$ through the point $X$, and a conic in the plane $[\mP]$. This corresponds in $\PG(2,q^2)$ to an $\li$-Baer pencil with vertex $\bar P$ that meets $[\B]$ in the two Baer sublines $\mP\cap\B$ and $m$.

 As there are $q+1$ choices for the point $X$ on $[T]$, there are $(q+1)(q-1)$ 3-spaces about $[P]$ that meets $[\B]$ in a twisted cubic, and $q+1$ that meet $[\B]$ in a line  and a conic, giving the required number of Baer pencils. \end{proof}

The next result shows that a non-degenerate conic in $\B$ lies in a unique $\li$-Baer pencil, and describes the relationship between the conic and the vertex of the pencil. 

\begin{theorem}\Label{cor:tgt-baby}
 Let $\B$ be a  Baer subplane in $\PG(2,q^2)$ tangent to $\li$ at the point $\bar T=\B\cap\li$ and let $\C$ be
 a non-degenerate conic  in $\B$ with $\bar T\in\C$.
 Then $\C$ lies in a unique $\li$-Baer pencil $\K$. Moreover, 
 the vertex of $\K$  lies in the extension of $\C$ to $\PG(2,q^2)$.
\end{theorem}

\begin{proof} Let $\C$ be a non-degenerate conic in $\B$, with $\bar T=\B\cap\li\in\C$. As $\li$ is not a line of $\B$, it is not the tangent line of $\C$ at the point $\bar T$. Let $\Cplus$ be the extension of $\C$ to $\PG(2,q^2)$, then 
 $\li$ is a secant to $\Cplus$, so $\Cplus\cap\li=\{\bar T,\bar L\}$. We will show that $\C$ lies in a unique $\li$-Baer pencil $\K$ which has  vertex $\bar L$. 

We first show that any point $X\in\Cplus$ projects $\C$ onto a Baer subline. 
Without loss of generality, let $\C=\{P_\theta=(1,\theta,\theta^2)\st\theta\in{\Fqq}\cup\{\infty\}\}$, so 
$\Cplus=\{(1,\theta,\theta^2)\st\theta\in{{\mathbb F}_q}\cup\{\infty\}\}$.
Let $\omega\in{{\mathbb F}_{q^2}}\setminus{{\mathbb F}_q}$, so the point  
$X=(1,\omega,\omega^2)$ lies in $\Cplus\setminus\C$. The projection of the point 
$P_\theta$, $\theta\in{{\mathbb F}_q}\cup\{\infty\}$ from $X$ onto the line $\ell$ with equation $x=0$ is $P_\theta'=(0,1,\theta+\omega)$. That is, the projection of $\C$ from $X$ onto $\ell$ is the  set $\{(0,1,\omega)+\theta(0,0,1)\st\theta\in{{\mathbb F}_q}\cup\{\infty\}\}$, which is a Baer subline.

We next show that $\C$  lies in a unique $\li$-Baer pencil.
By Result~\ref{cath-conic}, in $\PG(4,q)$, $[\C]$ is a twisted cubic meeting the spread line $[T]$ in one point, and $[\C]$  lies in a 3-space $\Pi$ that meets $[T]$ in exactly one point. Hence $\Pi$ contains a unique spread line $[P]$, with $\bar P\neq \bar T$. By Lemma~\ref{lemma-3-Baer}, $\Pi$ corresponds to an $\li$-Baer pencil $\K$ with vertex $\bar P$, so $\C$ lies in the pencil $\K$. 
If $\C$ were in two $\li$-Baer pencils $\K,\K'$, then $[\C]$ would lie in two 3-spaces $\Pi_\K$, $\Pi_{\K'}$, which is not possible.

Hence $\C$ lies in a unique $\li$-Baer pencil $\K$ with some vertex $\bar P\in\li$.  Further, as argued above, the point   $\bar L\in\Cplus\cap\li$ projects $\C$ onto a Baer subline, and so $\C$ lies in an $\li$-Baer pencil with vertex $\bar L$. Thus $\bar P=\bar L$ as required. 
\end{proof}

         The $\li$-Baer pencils give rise to partitions of the affine points of  a tangent Baer subplane into 
$q$ conics: one degenerate and $q-1$ non-degenerate.

\begin{corollary}\Label{thm:partition}
Let $\B$ be a  Baer subplane in $\PG(2,q^2)$ tangent to $\li$ at the point $\bar T=\B\cap\li$.  For each line $m$ of $\B$ through $\bar T$ and point $\bar P\in\li$, $\bar P\neq \bar T$, there is a set of $q$ $\li$-Baer pencils with  vertex $\bar P$ that partition the affine points of $\PG(2,q^2)$; and partition the affine points of $\B$ into $q$ conics through $\bar T$, one being degenerate. Moreover, the extension of each of these conics to $\PG(2,q^2)$ contains the point $\bar P$ (see Figure~\ref{conics-partition}).
\end{corollary}

\begin{figure}[h]\begin{center}

\hspace*{3cm}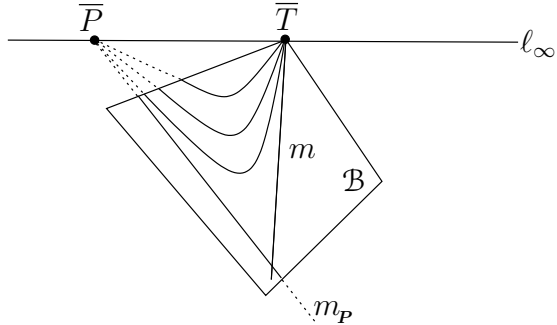
\caption{A partition of $\B\setminus \bar T$  into $q$ conics through $\bar T$}
\label{conics-partition}\end{center}
\end{figure}

\begin{proof}
The proof of Theorem~\ref{thm:partition-intro}
gives a construction for these partitions. The line $m$ corresponds in $\PG(4,q)$ to a line $[m]$ that meets the spread line $[T]$ in a point $X$. Let $\L$ be 
 the set of $q$ $3$-spaces  of $\PG(4,q)\setminus\si$ containing the plane $\alpha=\langle X,[P]\rangle$. These 3-spaces partition the affine points of $\PG(4,q)$ and hence partition the affine points of $[\B]$. 
 As argued in the proof of Theorem~\ref{thm:partition-intro}, one of the 3-spaces in $\L$ gives rise in $\PG(2,q^2)$ to  two lines in $\B$, and the remaining $q-1$ give rise to non-degenerate conics of $\B$ containing $\bar T$. By Theorem~\ref{cor:tgt-baby}, the extension of these conics to $\PG(2,q^2)$   contains the point $\bar P$. 
 \end{proof}

\section{Specialness and Baer sublines and subplanes}\Label{sec:special-other-Baer}

Parts 4 and 5  of Result~\ref{BB-Baer} illustrate that the concept of $g$-specialness is important in the Bruck-Bose representation of Baer substructures. In this section we discuss how parts 1 and 3 of Result~\ref{BB-Baer} relate to the notion of specialness.

Let $b$ be a Baer subline of $\li$, then by Result~\ref{BB-Baer}(1), in $\PG(4,q)$, $[b]$ is a regulus contained in the regular spread $\S$. Hence in $\PG(4,q^2)$, the transversals $g,g^q$ of $\S$ are lines of the regulus opposite to $\bstar$. That is, the regulus $[b]$ is closely related to the transversals of $\S$. There is another way to express this relationship.

\begin{theorem} \begin{enumerate}
\item Let $b$ be a Baer subline of $\li$ in $\PG(2,q^2)$. Then in the Bruck-Bose representation in $\PG(4,q)$, each non-degenerate conic contained in the regulus $[b]$ is a $g$-special conic.
\item Conversely, every $g$-special conic in $\si$ lies in a unique regulus of $\S$, and so corresponds to a Baer subline of $\li$.
\end{enumerate}
\end{theorem}

\begin{proof}
Let $b$ be a Baer subline of $\li$ in $\PG(2,q^2)$. By Result~\ref{BB-Baer}(1), in $\PG(4,q)$, $[b]$ is a regulus contained in the regular spread $\S$. There are $q^3-q$ planes of $\si$ that meet the regulus $[b]$ is a non-degenerate conic, namely the planes that do not contain a line of $[b]$. Let  $\alpha$ be such a plane, so $\alpha$ contains a unique spread line $[L]$, and $\C=[b]\cap \alpha$ is a non-degenerate conic. In $\PG(4,q^2)$, 
the transversal  $g$ meets each extended spread line, and so $g$ meets at least three lines of the extended regulus $\bstar$, hence $g$ is a line of the opposite regulus. In particular, each point of $g$ lies on one line of $\bstar$. Now 
$\C\star$ is the exact intersection $\bstar\cap\alphastar$, and $\alphastar$ meets $g$ in one point, hence $\C\star$ contains the points $g\cap\alphastar$, $g^q\cap\alphastar$, and so $\C$ is a $g$-special conic. 

Conversely, let $\C$ be a $g$-special conic in $\si$. So $\C$ lies in a plane $\alpha$, moreover, $\alpha$ contains a spread line $[L]$, and in $\PG(4,q^2)$, $\C\star$ contains the points $X=g\cap\Lstar$, $X^q=g^q\cap\Lstar$. 
Let $\K$ be the set of lines of $\S$ that meet $\C$, we need to show that $\K$ is a regulus. Let $[P_1],[P_2],[P_3]$ be three lines of $\K$ and let $\R$ be the unique regulus containing the three lines. By the argument above, $\D=\R\cap\alpha$ is a $g$-special conic, and $\D\star$ contains the points $X,X^q$. So $\C\star$, $\D\star$ have five points in common, namely $X,X^q,[P_i]\cap\alpha$, $i=1,2,3$. Hence $\C\star=\D\star$ and so $\K=\R$. That is, the points of $\C$ lie on lines of a regulus of $\S$, which by Result~\ref{BB-Baer} corresponds to a Baer subline of $\li$ in $\PG(2,q^2)$. 
\end{proof}

Furthermore, the regulus $[b]$ has a relationship to the lines in the hyperbolic congruence of $g,g^q$. 

\begin{theorem}\Label{thm:Baerline-trans} 
Let $b$ be a Baer subline of $\li$, and let $\bar P,\, \bar Q\in\li$ be conjugate with respect to $b$.
 Then  in $\PG(4,q^2)$, the lines $\gPQ$, $\gQP$ are lines of the  regulus  $\bstar$. 
 \end{theorem}
 
 \begin{proof}
 Let $\bar P,\, \bar Q$ be two points of $\li$ that are conjugate with respect to a  Baer subline $b\subset \li$. 
By Result~\ref{BB-Baer}, in $\PG(4,q)$, $[b]$ is a  regulus of $\S$.
By Result~\ref{res:circle}, the unique partition of $\S\setminus\{[P],[Q]\}$ into reguli contains the regulus $[b]$; and   in $\PG(4,q^2)$, the lines $\gPQ,\gQP$ meet each line of the regulus opposite to $\bstar$. Hence  the lines $\gPQ,\gQP$
are lines of the regulus $\bstar$.
\end{proof}

\begin{remark} {\rm Given a Baer subline $b$ of $\li$, the points of $\li\setminus\{b\}$ can be partitioned into pairs of points $\{\bar P_i,\bar Q_i\}$ which are 
conjugate with respect to $b$. Hence the $q^2-q$ lines $\gPiQi,(\gPiQi)^q$ are exactly the lines of $\PG(4,q^2)$ in the regulus $\bstar$ that are not lines of $\PG(4,q)$.
}
\end{remark}

We now consider a Baer subplane $\B$ of $\PG(2,q^2)$ secant to $\li$. By Result~\ref{BB-Baer}, $[\B]$ is a plane of $\PG(4,q)$, and the line $[\B]\cap\si$ meets $q+1$ lines of $\S$ which form a regulus denoted by $\R$. As noted above, in $\PG(4,q^2)$, the transversals $g,g^q$ are lines of the regulus opposite to $\R$. 
Moreover, by Theorem~\ref{thm:Baerline-trans} 
the extended regulus $\R\star$ contains the line $\gPQ$ where the corresponding points $\bar  P,\bar  Q\in\li$ are conjugate with respect to $\B$.

\begin{corollary}\Label{cor:Baerplane-trans} 
Let $\B$ be a Baer subplane of $\PG(2,q^2)$ that is secant to $\li$, and let $\bar  P,\bar  Q\in\li$ be conjugate with respect to $\B$.
 Then  in $\PG(4,q^2)$, the lines $\gPQ$, $\gQP$ meet the plane $\Bstar$.
\end{corollary}

\section{Conics of $\PG(2,q^2)$}\Label{sec:adult-conic}

In \cite{BJQ}, it is shown that   a non-degenerate conic $\mathscr O$ in $\PG(2,q^2)$ corresponds in $\PG(4,q)$ to   the    intersection of two quadrics. 
 Moreover, this correspondence is exact-at-infinity: that is, an affine point $A\in\PG(2,q^2)\setminus\li$ lies in $\mathscr O$ if and only if  the affine point $[A]\in\PG(4,q)\setminus\si$ lies in $[\mathscr O]=\Q_1\cap\Q_2$; and  a point $\bar T\in\li$ lies in $\mathscr O$ if and only if the spread line $[T]$ is contained in $[\mathscr O]=\Q_1\cap\Q_2$. So the set $[\mathscr O]=\Q_1\cap\Q_2$ meets $\si$ either in the empty set, or in  1 or 2 spread lines.
We determine the relationship of $[\mathscr O]$ with the transversals $g,g^q$ of the regular spread $\S$. 

The arguments used  are coordinate based. A conic $\mathscr O$ has equation $f(x,y,z)=0$ where $f$ is a homogeneous equation of degree two over ${{\mathbb F}_{q^2}}$. Using the Bruck-Bose map, this can be written as 
$f_\infty(x_0,x_1,y_0,y_1,z)+\tau f_0(x_0,x_1,y_0,y_1,z)=0$ where $f_\infty=0$, $f_0=0$ are homogeneous quadratic equations over ${{\mathbb F}_q}$, and so they are equations of quadrics $\Q_\infty$, $\Q_0$ in $\PG(4,q)$, hence $[\mathscr O]=\Q_\infty\cap\Q_0$.
Moreover, $[\mathscr O]$ is contained in the pencil of quadrics $\{\Q_t =t  \Q_\infty+\Q_0, t \in{{\mathbb F}_q}\cup\{\infty\}\}$ where
$\Q_t$ has equation $f_t =t  f_\infty+f_0=0$.
 There is a natural extension to $\PG(4,q^2)$ and to $\PG(4,q^4)$, namely $\mathscrOstar=\Qonestar\cap\Qtwostar$ and $\mathscrOstarstar=\Qonestarstar\cap\Qtwostarstar$.
In order to study subconics in Baer subplanes, we will need a full analysis of 
 how these sets meet the hyperplane at infinity, which we give in this section.
We first show that none of the quadrics $\Qtstar$, $t\in {{\mathbb F}_q}\cup\{\infty\}$, contain $g$, and so $\mathscrOstar$ does not contain $g$.

 \begin{theorem}\Label{adult-conic-g} Let $\mathscr O$ be a non-degenerate conic in $\PG(2,q^2)$, so $[\mathscr O]=\Q_\infty\cap\Q_0$. In $\PG(4,q^2)$, the  quadric $\Qtstar=t \Qonestar+\Qtwostar$, $t \in{{\mathbb F}_q}\cup\{\infty\}$, meets $g$ in $0$, $1$ or $2$ points, according to whether $\mathscr O$ meets $\li$ in $0$, $1$ or $2$ points respectively.\end{theorem}

\begin{proof}  
Consider first the case when $\mathscr O$ is tangent to $\li$.  The group $\PGL(3,q^2)$ is transitive on  non-degenerate conics, and the subgroup fixing a  non-degenerate conic $\mathscr O$ is transitive on the tangent lines of $\mathscr O$. Hence we can  without loss of generality,  prove the result for the conic $\mathscr O$ of equation
$y^2=xz$ in $\PG(2,q^2)$, which meets $\li$ in one (repeated) point $\bar T=(1,0,0)$.
The affine point $(x,y,1)=(x_0+x_1\tau,y_0+y_1\tau,1)$ is on $\mathscr O$ if $(y_0+y_1\tau)^2=x_0+x_1\tau$, that is
$(y_0^2+y_1^2t_0-x_0)+(y_1^2t_1+2y_0y_1-x_1)\tau=0$.
The solutions $(x_0,x_1,y_0,y_1,1)\in\PG(4,q)$  to this are the affine points in $[\mathscr O]$. 
That is, $[\mathscr O]$ is  the intersection of the two quadrics $\Q_\infty,\Q_0$ with homogeneous equations $f_\infty=0$, $f_0=0$ respectively, where 
\begin{eqnarray}\label{eqn:quad-tgt}
f_\infty=y_0^2+y_1^2t_0-x_0z\quad\mbox{and}\quad  f_0=y_1^2t_1+2y_0y_1-x_1z.
\end{eqnarray}
Note that the intersection $[\mathscr O]=\Q_\infty\cap\Q_0$ is exact on $\si$: both $\Q_\infty$ and $\Q_0$ contain the spread line $[T]=\{(a,b,0,0,0)\st a,b\in{{\mathbb F}_q}\}$, and these are the only points of $\si$ contained in both $\Q_\infty$ and $\Q_0$. Also note that  in $\PG(4,q^2)$, $\Qonestar$ and $\Qtwostar$ both contain the extended spread line $\Tstar$, and so both contain at least one point  of $g$, namely $\Tstar\cap g=A_0$. Also, $[\mathscr O]$ lies in the pencil of quadrics $\{\Q_t=t  \Q_\infty+\Q_0\st
t \in{{\mathbb F}_q}\cup\{\infty\}\}$ where
$\Q_t$ has equation $f_t =t  f_\infty+f_0=0$. 
Recall that the transversal $g$ of $\S$ consists of the points  $G_{\beta}=\beta A_0+A_1=\big(\beta\tau^q,\ -\beta,\ \tau^q,\ -1,\ 0\big)$ for $\beta\in{{\mathbb F}_{q^2}}\cup\{\infty\}$. 
For $\beta\in{{\mathbb F}_{q^2}}$, we have  $f_\infty(G_\beta)=\tau^q(\tau^q-\tau)$ and $f_0(G_\beta)=\tau-\tau^q$. Let $f_t=t f_\infty+f_0$, then $f_t(G_\beta)=(\tau^q-\tau)(t \tau^q-1)$ which is never zero when $t \in{{\mathbb F}_q}$. Hence  $G_\infty=A_0$  is the only point of $g$ contained in the quadric $\Q_t $. Similarly, $A_0^q$  is the only point of the (other) transversal $g^q$ contained in the quadric $\Q_t $. 
That is, when $\mathscr O$ is tangent to $\li$, the quadrics $\Qtstar$ each meet $g$ in one point, namely $\Tstar\cap g=A_0$.
 A similar argument using the conic with equation $
f(x,y,z)=x^2-\delta y^2+z^2
$, $\delta\in{{\mathbb F}_{q^2}}\setminus\{0\}$ for $q$   odd, and  $\delta x^2+y^2+z^2+yx=0$, $\delta\in{{\mathbb F}_{q^2}}$ for $q$ even completes the other cases. 
\end{proof}

The  proof of Theorem~\ref{adult-conic-g}, and the  1-1  correspondence between  points $\bar P$ of $\li$ and points $\Pbar=[P]\star\cap g$ of the transversal $g$, allow us to identify the points of the quadric $\Qtstar$ on  $g$.

\begin{corollary} \Label{cor:PcorrPsigma}
Let $\mathscr O$ be a non-degenerate conic in $\PG(2,q^2)$, then 
\begin{enumerate}
\item $\bar P\in\mathscr O\cap\li$  if and only if in $\PG(4,q^2)$, $\Pbar\in g$;
\item  $\bar P$ is a point in the intersection of the extension of $\mathscr O$ and the extension of $\li$  to $\PG(2,q^4)$ if and only if in $\PG(4,q^4)$,
  $P\in \gstar\setminus g$.
  \end{enumerate}
\end{corollary}

Next we consider the set $[\mathscr O]$ extended to $\PG(4,q^2)$ and $\PG(4,q^4)$, and determine the exact intersection with the  hyperplane at infinity. 

\begin{theorem}\Label{thm:Ccapsi}
 Let $\mathscr O$ be a non-degenerate conic in $\PG(2,q^2)$.

\begin{enumerate}
\item Suppose $\mathscr O$ is secant to $\li$, so $\mathscr O\cap \li=\{\bar P,\, \bar Q\}$, then 
\begin{enumerate}
\item in $\PG(4,q)$, $[\mathscr O]\cap\si=\{[P],[Q]\}$;
\item in $\PG(4,q^2)$, $\mathscrOstar\cap\sistar=\{\Pstar,\Qptstar,\gPQ,\gQP\}$; 
\item in $\PG(4,q^4)$, $\mathscrOstarstar\cap\sistarstar=\{\Pstarstar,\ \Qptstarstar,\ (\gPQ)^\blackstar,\ (\gQP)^\blackstar \}$.
\end{enumerate}
\item Suppose $\mathscr O$ is tangent to $\li$, so $\mathscr O\cap \li=\{\bar P\}$, then 
\begin{enumerate}
\item in $\PG(4,q)$, $[\mathscr O]\cap\si=\{[P]\}$;
\item in $\PG(4,q^2)$, $\mathscrOstar\cap\sistar=\{\Pstar\}$,
\item in $\PG(4,q^4)$, $\mathscrOstarstar\cap\sistarstar=\{\Pstarstar\}$.
\end{enumerate}
\item Suppose $\mathscr O$ is exterior to $\li$, so in the extension to $\PG(2,q^4)$,  the extension of $\mathscr O$ meets the extension of  $\li$ in  two points $\{\bar P,\bar P^{q^2}\}$. Then 
\begin{enumerate}
\item in $\PG(4,q)$, $[\mathscr O]\cap\si=\emptyset$;
\item in $\PG(4,q^2)$, $\mathscrOstar\cap\sistar=\emptyset$;
\item in $\PG(4,q^4)$, $\mathscrOstarstar\cap\sistarstar=\{\elllP,\ \elllP^q,\ \elllP^{q^2},\ \elllP^{q^3}\}$, where $\elllP=PP^q$.
\end{enumerate}
\end{enumerate}
\end{theorem}

\begin{proof} As noted above,  $[\mathscr O]=\Q_\infty\cap\Q_0$, for quadrics $\Q_\infty,\Q_0$, and this correspondence is exact, so $[\mathscr O]$ meets $\si$ in either  the empty set, or in  1 or 2 spread lines (corresponding  respectively to $\mathscr O$ meeting $\li$ in 0, 1 or 2 points).  
The cases $\mathscr O$ tangent, secant and exterior to $\li$, $q$ odd and even, are proved separately using  the same conic equations as in the proof of Theorem~\ref{adult-conic-g}.
We omit the calculations, noting that we rely on \cite[Table 2]{BruenHirschfeld} to show that the intersection of the two quadrics in the 3-space $\sistar$ is a set of four lines, possibly repeated. 
\end{proof}

We have shown that  in $\PG(4,q^2)$, the set $\mathscrOstar$ contains an extended spread line $\Pstar$ if and only if in $\PG(2,q^2)$, the point $\bar P\in\mathscr O\cap\li$. We will need the next corollary which considers whether the set $\mathscrOstar$ can contain a point of any other extended spread line.

\begin{corollary}\Label{adult-conic-T} Let $\mathscr O$ be a non-degenerate conic in $\PG(2,q^2)$. Let $\bar L$ be a point of $\li$ not in $\mathscr O$. In $\PG(4,q^2)$, the corresponding extended spread line $\Lstar$ is disjoint from $\mathscrOstar$.
\end{corollary}

\begin{proof} If $\mathscr O$ 
is secant to $\li$, so $\mathscr O\cap\li=\{\bar P,\, \bar Q\}$, 
then by Theorem~\ref{thm:Ccapsi}, $\mathscrOstar\cap\sistar$ consists of the four lines $\Pstar,\Qptstar,\gPQ,\gQP$. 
Let $\Lstar$ be an extended spread line, $\bar L\neq \bar P,\, \bar Q$. Then $\Lstar, \Pstar,\Qptstar,\gPQ,\gQP$ are all lines of the hyperbolic congruence of $g,g^q$, and so do not meet off $g,g^q$, and hence are mutually skew. So $\Lstar\cap\mathscrOstar=\emptyset$.
If $\mathscr O$ is tangent to $\li$, then by  Theorem~\ref{thm:Ccapsi}, $\mathscrOstar\cap\sistar=\Pstar$. Hence $\mathscrOstar$ meets no other spread line. 
If $\mathscr O$ is exterior to $\li$, then by Theorem~\ref{thm:Ccapsi}, 
$\mathscrOstar\cap\sistar=\emptyset$, so  
 $\mathscrOstar$ contains no point on any extended spread line, as required.
\end{proof}

\section{Conics of Baer subplanes}\Label{sec:subconics}

In this section we improve Result~\ref{cath-conic} by characterising the normal rational curves of $\PG(4,q)$  that correspond to  conics of a Baer subplane of $\PG(2,q^2)$. In particular, we  show that if $\C$ is a conic contained in a tangent Baer subplane $\B$ of $\PG(2,q^2)$, then in $\PG(4,q)$, the corresponding 3- or 4-dimensional normal rational curve $[\C]$ is $g$-special. Further, we show that any $g$-special  3- or 4-dimensional normal rational curve in $\PG(4,q)$ corresponds to a conic in a Baer subplane of $\PG(2,q^2)$.

\subsection{$\Fqq$-conics and $\Fq$-conics}

In this section we show that the notion of specialness is also intrinsic to the Bruck-Bose representation of conics in Baer subplanes. 
First we introduce some notation to easily distinguish between conics in $\PG(2,q^2)$ and conics contained in a Baer subplane. An {\em $\Fqq$-conic} in $\PG(2,q^2)$ is a non-degenerate conic of $\PG(2,q^2)$. 
 Note that an $\Fqq$-conic meets a Baer subplane $\B$ in either 0, 1, 2, 3 or 4 points, or in a non-degenerate conic of $\B$. 
We define an {\em $\Fq$-conic} of $\PG(2,q^2)$ to be a non-degenerate conic in a Baer subplane of $\PG(2,q^2)$.
For the remainder of this article, $\C$ will denote an $\Fq$-conic. Further,  we denote the {\em unique} $\Fqq$-conic containing $\C$ by {\rm $\Cplus$}. 
An $\Fqq$-conic contains many $\Fq$-conics.

\begin{lemma}\Label{adult-baby}
Let $\mathscr O$ be an $\Fqq$-conic in $\PG(2,q^2)$.  Any three points of $\mathscr O$  lie in a unique $\Fq$-conic that is contained in $\mathscr O$, so there are $q(q^2+1)$ $\Fq$-conics contained in $\mathscr O$. 
\end{lemma}

\begin{proof}
The $\Fqq$-conic  $\mathscr O$ is   equivalent to the line $\ell\cong\PG(1,q^2)$,  and subconics of $\mathscr O$ are   equivalent  to Baer sublines of $\ell$. Since three points of $\ell$ lie in a unique Baer subline of $\ell$, three points of $\mathscr O$ lie in a unique subconic $\C$.
As $\C$ is a normal rational curve over ${{\mathbb F}_q}$, 
 by \cite[Theorem 21.1.1]{hirs91} there is a homography $\phi$ that maps $\C$ to $\C'=\phi(\C)= \{(1,\theta,\theta^2)\st\theta\in{{\mathbb F}_q}\cup\{\infty\}\}$. As $\C'$ lies in the Baer subplane $\B'=\PG(2,q)$, $\C$ lies in the Baer subplane $\phi^{-1}(\B')$, that is, $\C$ is an $\Fq$-conic.  Straightforward counting shows that the number of $\Fq$-conics in $\mathscr O$ is $(q^2+1)q^2(q^2-1)/(q+1)q(q-1)=q(q^2+1)$.
\end{proof}

\begin{remark}\Label{remark-subconic-conic}
{\rm Let $\C$ be an $\Fq$-conic in $\PG(2,q^2)$, $q>4$, so there is a unique $\Fqq$-conic {\rm $\Cplus$}  with $\C\subset \Cplus$.  Then in $\PG(4,q)$, $[\C]\subset[\Cplus]$. This is clearly true for the affine points. For the points at infinity, we recall Remark~\ref{subconic-exact}, if $\bar T\in\C\cap\li\subseteq\Cplus\cap\li$, then $[\C]$ meets the spread line $[T]$ in a point, while $[\Cplus]$ contains  the spread line $[T]$. 
}\end{remark}

\subsection{Conics in secant Baer subplanes}\Label{sec:secant-Baer-conic}

In this section we consider 
the Bruck-Bose representation of $\Fq$-conics in secant Baer subplanes of $\PG(2,q^2)$, in particular looking at the relationship with the lines of the hyperbolic congruence of $g,g^q$.

\begin{theorem}\Label{lem:sect-conic}
Let $\C$ be an $\Fq$-conic in a Baer subplane $\B$ secant to $\li$. The $\Fqq$-conic {\rm {\rm $\Cplus$}} meets $\li$ in two points $\bar P,\, \bar Q$ (possibly equal). In $\PG(4,q)$, $[\C]$ is a non-degenerate conic in the plane $[\B]$, and {\rm $[\Cplus]\cap\si$}  is the two spread lines $[P]$, $[Q]$.
\begin{enumerate}
\item If $\bar P= \bar Q$, then $\bar P\in\B$, and $[\C]$ meets $\si$ in one point  $[P]\cap[\B]$.
\item If $\bar P\neq \bar Q$ and 
 $\bar P,\, \bar Q\in\B$, then $[\C]$ meets $\si$ in two points 
$[P]\cap[\B]$ and $[Q]\cap[\B]$.
\item  If $\bar P\neq \bar Q$ and    $\bar P,\, \bar Q\notin\B$, then  $[\C]$ is a $(\gPQ)$-special conic.
\end{enumerate}
\end{theorem}

\begin{proof}
By Results~\ref{BB-Baer} and~\ref{cath-conic}, in $\PG(4,q)$, 
$[\B]$ is a plane, and $[\C]$ is a conic in $[\B]$. Parts 1 and 2 follow immediately from the Bruck-Bose definition. For part 3, the set $[\Cplus]$ contains the spread lines $[P]$ and $[Q]$. The set 
$[\B]$ is a plane, and the line $m=[\B]\cap\si$ meets $q+1$ spread lines, but does not meet the spread lines $[P]$, $[Q]$. 
The set  $[\C]$ is a non-degenerate conic in 
 $[\B]$ which does not meet $m$, and in the extension to $\PG(4,q^2)$, $\Cstar$  meets $\sistar$ in two points  of the line $\mstar=\Bstar\cap\sistar$.  
In $\PG(2,q^2)$, we have $\C=\B\cap\Cplus$, and in $\PG(4,q)$,   $[\C]=[\B]\cap[\Cplus]$. Moreover, in $\PG(4,q^2)$, $\Cstar=\Bstar\cap\Cplusstar$, 
hence $\Cstar\cap\sistar=\{\Bstar\cap\sistar\}\cap\{\Cplusstar\cap\sistar\}$. By Theorem~\ref{thm:Ccapsi}, this  is equal to $\{\mstar\}\cap\{g,g^q,\gPQ,\gQP\}$. Now $\mstar$ does not meet $g$ (or $g^q$) as the only lines of $\si$ whose extension meets $g$ are the lines of $\S$. Hence the two points of $\Cstar\cap\sistar$ lie in $\gPQ$ and $\gQP$, that is, $[\C]$ is a $(\gPQ)$-special conic of $\PG(4,q)$. 
\end{proof}

We now characterise $\Fq$-conics in secant Baer subplanes by showing that the converse is true. 

\begin{theorem} In $\PG(4,q)$, let $\alpha$ be a plane not containing a spread line, and let $\N$ be a non-degenerate conic in $\alpha$. 
\begin{enumerate}
\item In $\PG(2,q^2)$, there is a secant Baer subplane $\B$ with $[\B]=\alpha$, and an $\Fq$-conic $\C$ in $\B$ with $[\C]=\N$.
\item If $\N$ meets $\si$ in a point of the spread line $[T]$, then $\bar T\in\C$.
\item If $\N$ is a $(\gPQ)$-special conic, then the $\Fqq$-conic  {\rm $\Cplus$} containing $\C$  meets $\li$ in the points $\bar P,\, \bar Q$.
\end{enumerate}
\end{theorem}

\begin{proof}
 Parts 1 and 2  follow from Result~\ref{BB-Baer}.
For part 3, in $\PG(4,q)$, let $\N$ be a $(\gPQ)$-special conic of $\PG(4,q)$ lying in a plane  $\alpha$ that does not contain a spread line. By part 1, $[\B]=\alpha$ and $[\C]=\N$ where in $\PG(2,q^2)$, $\B$ is a secant Baer subplane containing the $\Fq$-conic $\C$.
 As $\N$ is a $(\gPQ)$-special conic, $\N\cap\si=\emptyset$, and in $\PG(4,q^2)$, $\Nstar$ is a conic which meets the line $\alpha\cap\sistar$ is two points,  one lying on each of $\gPQ$ and $\gQP$.  As $\N\cap\si=\emptyset$, in $\PG(2,q^2)$, the $\Fq$-conic  $\C$ does not meet $\li$, so the $\Fqq$-conic $\Cplus$ meets $\li$ in two points $\bar A,\bar B\notin\B$. By Theorem~\ref{lem:sect-conic}, $[\C]=\N$ is a $(\gAB)$-special conic. Hence $\{\bar A,\bar B\}=\{\bar P,\, \bar Q\}$, so $\Cplus\cap\li=\{\bar P,\, \bar Q\}$ as required. 
\end{proof}

\subsection{Conics in tangent Baer subplanes}

We now consider a Baer subplane $\B$ that is tangent to $\li$ and  look at $\Fq$-conics in $\B$. There are  two cases to consider, namely whether the $\Fq$-conic  contains the point $\B\cap\li$ or not. In each case we generalise Result~\ref{cath-conic} by showing that the corresponding normal rational curve of $\PG(4,q)$  is $g$-special. Further, we characterise all $g$-special  normal rational curves in $\PG(4,q)$  as corresponding to $\Fq$-conics in a tangent Baer subplane.

\subsubsection{Conics in $\B$ containing the point $\bar T=\B\cap\li$}

We first look at an $\Fq$-conic $\C$ in a tangent Baer subplane $\B$, with $\B\cap\li$ in $\C$. 

%

\begin{theorem}\Label{thm-tgt-conic-T-1} 
In $\PG(2,q^2)$, $q>5$, let $\B$ be a tangent Baer subplane and $\C$ an $\Fq$-conic in $\B$ containing the point $\B\cap\li$. Then in $\PG(4,q)$, $[\C]$ is  a $g$-special twisted cubic.
\end{theorem}


\begin{proof} Let $\B$ be a Baer subplane of $\PG(2,q^2)$ that is tangent to $\li$ in the point $\bar T=\B\cap\li$. Let $\C$ be an $\Fq$-conic of $\B$ that contains $\bar T$. By Result~\ref{cath-conic}, in $\PG(4,q)$, $\C$ corresponds to a twisted cubic $[\C]$ that lies in a $3$-space denoted $\PiC$. By Result~\ref{3-space-meets-ruled}, $\PiC$ meets the ruled cubic surface  $[\B]$ in exactly the twisted cubic $[\C]$. 
We   show that $[\C]$ is  $g$-special. 
By Lemma~\ref{lemma-3-Baer}, the $3$-space $\PiC$ corresponds to an $\li$-Baer pencil $\K$ of $\PG(2,q^2)$ that meets $\B$ in $\C$. By Theorem~\ref{cor:tgt-baby}, $\K$ has vertex $\bar P\in\Cplus$. Hence $\PiC$ contains the spread line $[P]$ (and this is the only spread line in $\PiC$). 
Consider the extension of $\PG(4,q)$ to $\PG(4,q^2)$.
Note that Lemma~\ref{lem:nrc-extn} can be generalised to a 3-dimensional normal rational curve when $q>5$. Hence as $[\B]$ is the intersection of three quadrics \cite{FFA},  we have $\Cstar\subset\Bstar$ in $\PG(4,q^2)$. Thus by Corollary~\ref{lem:tc-brs}, the twisted cubic $\Cstar$ contains a unique point of each generator line of the ruled cubic surface $\Bstar$. By Result~\ref{BB-Baer}, $[\B]$ is $g$-special, so the transversal lines $g,g^q$ of the regular spread $\S$ are generator lines of the extended ruled cubic surface $\Bstar$. Hence $\Cstar$ contains a point of $g$ and $g^q$. Thus $\Cstar$ contains the points corresponding to the vertex of $\K$, that is, the point $g\cap \PiCstar=g\cap\Pstar=P$ and $P^q$.  
That is, the twisted cubic $[\C]$ is $g$-special. 
\end{proof}

The converse of Theorem~\ref{thm-tgt-conic-T-1} is also true.

\begin{theorem}\Label{conv-tgt}
A $g$-special twisted cubic in $\PG(4,q)$ corresponds to an $\Fq$-conic in some tangent Baer subplane of $\PG(2,q^2)$. 
\end{theorem}

\begin{proof} 
Let $\N$ be a $g$-special twisted cubic in $\PG(4,q)$, so in $\PG(4,q^2)$,  $\Nstar$ meets the transversal $g$ of $\S$ in a point $R$, and meets $g^q$ in the point $R^q$.  The line $RR^q$ meets $\si$ in a spread line denoted $[R]$, corresponding to the point  $\bar R\in\li$. 
Let $\PiN$ be the 3-space containing $\N$, and recall that a twisted cubic meets a plane in three points, possibly repeated, or in an extension.  As   $\N$ meets the plane $\pi=\PiN\cap\si$  in two points, $R,R^q$ over ${{\mathbb F}_{q^2}}$,  $\N$ meets $\pi$ in one point $X$ over ${{\mathbb F}_q}$. 
Let $[T]$ be the spread line containing the point $X$,  so $[T]\notin\PiN$. Let $[A],[B],[C]$ be three affine points of $\N$, and let $\alpha=\langle [A],[B],[C]\rangle$.

As $\alpha$ lies in the 3-space $\PiN$, if $\alpha$ contained a spread line, it would contain $[R]$. However, if $\alpha$ contains $[R]$, then the plane $\langle [A],[B],[R]\rangle\star$ would contain four points of $\Nstar$, namely $[A],[B],R, R^q$, a contradiction. 
If $\alpha$  contained the point $X$, then $\alpha$ would contain four points of  $\N$, namely $X,[A],[B],[C]$, a contradiction. 
Hence $\alpha$ corresponds to a Baer subplane $\B_\alpha$ of $\PG(2,q^2)$ that is secant to $\li$, with $\bar T\notin \B_\alpha$.
 Hence  the  points $\{\bar T,A,B,C\}$ form a quadrangle  and so lie in a unique Baer subplane denoted $\B$.
As $\B_\alpha$ is the unique Baer subplane containing $A,B,C$ and secant to $\li$, we have $\B\neq\B_\alpha$, and $\B$ is tangent to $\li$ at the point $\bar T$. 

In $\PG(4,q)$, $[\B]$ is a ruled cubic surface with line directrix $[T]$. As $X,[A],[B],[C]$ are points of $\N$, no three are collinear, so  $[A],[B],[C]$ lie on distinct generators of $[\B]$. Recall that $\PiN$ does not contain $[T]$, so by Result~\ref{3-space-meets-ruled}, $\PiN$ meets $[\B]$ in a twisted cubic, denoted $\N_1$. The argument in the proof of Theorem~\ref{thm-tgt-conic-T-1} shows that in the quadratic extension, $\Nonestar$ contains the points 
$R$ and $R^q$. Hence $\Nstar$ and $\Nonestar$ share six points, and so are equal. That is, $\N$ is a $g$-special twisted cubic contained in a $g$-special ruled cubic surface, and $\N$ meets $\si$ in one point.

Straightforward counting shows that in $\PG(2,q^2)$, the number of $\Fq$-conics in $\B$ that contain $\bar T$ is $q^4-q^2$. By  Result~\ref{3-space-meets-ruled}, the number of $3$-spaces of $\PG(4,q)$ that meet the ruled cubic surface $[\B]$ in a twisted cubic is $q^4-q^2$. Hence they are in one to one correspondence. That is,
$\N$ corresponds to an $\Fq$-conic in the tangent Baer subplane $\B$ as required. 
 \end{proof}

Further, the proofs of Theorems~\ref{thm-tgt-conic-T-1} and~\ref{conv-tgt}  show that  the points of $g$ on a $g$-special twisted cubic correspond to the points on $\li$ contained in the  corresponding $\Fqq$-conic. 

\begin{corollary}\Label{thm-tgt-conic-T-2} 
Let $\C$ be an $\Fq$-conic in 
a tangent Baer subplane $\B$ in $\PG(2,q^2)$, $q>5$, with $\bar T=\B\cap\li\in\C$. The $\Fqq$-conic {\rm $\Cplus$}  meets $\li$ in a point $\bar P\neq \bar T$ if and only if in $\PG(4,q^2)$, the twisted cubic $\Cstar$ meets the transversals of $\S$ in  the points $P$, $P^q$. 
\end{corollary}

\subsubsection{Conics of $\B$ not containing the point $\bar T=\B\cap\li$}

We now look at an $\Fq$-conic $\C$ in a tangent Baer subplane $\B$, with $\B\cap\li$ not in $\C$. The $\Fqq$-conic {\rm $\Cplus$}  meets $\li$ in two distinct points (which may lie in $\PG(2,q^4)$). We show that if these two points lie in $\PG(2,q^2)$, then $[\C]$ is a $g$-special normal rational curve. Further, if the two points lie in the quadratic  extension of $\PG(2,q^2)$ to $\PG(2,q^4)$, then $[\C]$ is an $\gstar$-special normal rational curve.

\begin{theorem}\Label{smiley-conic}
In $\PG(2,q^2)$, $q>7$, let $\B$ be a Baer subplane tangent to $\li$ with $\bar T=\B\cap\li$. Let $\C$ be an $\Fq$-conic in $\B$, $\bar T\notin\C$. In $\PG(4,q)$, $[\C]$ is a $g$-special or $\gstar$-special 4-dimensional normal rational curve. 
\end{theorem}

\begin{proof} 
Let $\C$ be an $\Fq$-conic in $\B$ not through $\bar T=\B\cap\li$, and consider the $\Fqq$-conic $\Cplus$ containing $\C$.
Then either
(i) $\Cplus$ is secant to $\li$ and $\Cplus\cap\li$ is two distinct points $\bar P,\, \bar Q$; (ii) $\Cplus$ is tangent to $\li$ and $\Cplus\cap\li$ is a repeated point $\bar P=\bar Q$;  or (iii) $\Cplus$ is exterior to $\li$ and in $\PG(2,q^4)$, the extension of $\Cplus$ meets the extension of $\li$ in two points  $\bar P,\, \bar Q$ which are conjugate with respect to this  extension from $\PG(2,q^2)$ to $\PG(2,q^4)$, that is,  $\bar Q=\bar P^{q^2}$.  
By Result~\ref{cath-conic}, as $\bar T\notin\C$, in $\PG(4,q)$, $[\C]$ is a 4-dimensional normal rational curve lying on the $g$-special ruled cubic surface $[\B]$, and $[\C]$ does not meet $\si$. Thus it remains to show that in $\PG(4,q)$ $[\C]$ is a $g$-special or $\gstar$-special.
We will show that in an appropriate extension of $\PG(4,q)$, the extension of the  normal rational curve $[\C]$ contains the  points $\Pbar, \Qbar$ of the (possibly extended) transversal $g$, giving the $g$-special property.
Recall that a 4-dimensional 
normal rational curve of $\PG(4,q)$ meets the $3$-space $\si$ in four points, possibly repeated or in an extension. As $[\C]$ is disjoint from $\si$, either (a) in $\PG(4,q^2)$, $\Cstar$ meets $\sistar$ in four points of the form $X,X^q,Y,Y^q$, possibly $X=Y$; or (b) in $\PG(4,q^4)$,  $\Cstarstar$ meets $\sistarstar$ in four points of form $X,X^q,X^{q^2},X^{q^3}$.

 
In $\PG(2,q^2)$, we have $\C\subset\Cplus$, so as discussed in Remark~\ref{remark-subconic-conic}, in $\PG(4,q)$,  $[\C]\subset[\Cplus]$. By \cite[Cor 3.3]{BJQ},  $[\Cplus]$  is the exact intersection of two quadrics, so by Lemma~\ref{lem:nrc-extn}: in $\PG(4,q^2)$, $\Cstar\subset \Cplusstar$; and in $\PG(4,q^4)$, $\Cstarstar\subset \Cplusstarstar$. Similarly, as $[\C]\subset [\B]$   and $[\B]$ is the intersection of three quadrics \cite{FFA}, by Lemma~\ref{lem:nrc-extn},
 $\Cstar\subset\Bstar$ and $\Cstarstar\subset\Bstarstar$.
 In $\PG(2,q^2)$, we have $\C=\B\cap\Cplus$. As $[\C]$ is disjoint from $\si$, in $\PG(4,q)$, we have $[\C]=[\B]\cap[\Cplus]$.
 We need to determine  $\Cstar\cap\sistar= \Cplusstar\cap\Bstar\cap\sistar$  in $\PG(4,q^2)$, and  
 $\Cstarstar\cap\sistarstar= \Cplusstarstar\cap\Bstarstar\cap\sistarstar$ in $\PG(4,q^4)$.
 
 First we determine $\Bstar\cap\sistar$ and $\Bstarstar\cap\sistarstar$. In $\PG(2,q^2)$,  $\bar T\in \B$, so in $\PG(4,q)$, $[T]\subset[\B]$, and $[\B]\cap\si=[T]$. Hence in $\PG(4,q^2)$, $\Tstar\subset\Bstar$, and  in $\PG(4,q^4)$, $\Tstarstar\subset\Bstarstar$.
 By Result~\ref{BB-Baer},  $[\B]$ is a $g$-special ruled cubic surface, so the transversal lines $g,g^q$ lie in $\Bstar$. That is,  $\{\Tstar,g,g^q\}$ lie in $\Bstar$, and using Result~\ref{3-space-meets-ruled} in $\PG(4,q^2)$,   the 3-space $\sistar$ meets the ruled cubic surface $
\Bstar$ in exactly these three lines, so 
$\Bstar\cap\sistar= \{\Tstar,g,g^q\}$. Similarly, in $\PG(4,q^4)$, the 3-space $\sistarstar$ meets the ruled cubic surface $
\Bstarstar$ in the three lines $\{\Tstarstar,\gstar,\gqstar\}$.

Recall that Theorem~\ref{thm:Ccapsi} determines the  intersection $\Cplusstar\cap\sistar$ and $\Cplusstarstar\cap\sistarstar$ for  the three cases where {\rm $\Cplus$}  is (i) secant, (ii) tangent or (iii) exterior to $\li$  in $\PG(2,q^2)$. 
 For each case we determine $\Cplusstar\cap\Bstar\cap\sistar$ in $\PG(4,q^2)$ 
 and  $\Cplusstarstar\cap\Bstarstar\cap\sistarstar$ in $\PG(4,q^2)$.

 In case (i), {\rm $\Cplus$}  is secant to $\li$, so by Theorem~\ref{thm:Ccapsi}, $\Cplusstar\cap\sistar=\{\Pstar,\Qptstar,\Pbar \Qqbar,\ \Pqbar \Qbar\}$. 
Now  
 $\Bstar\cap\sistar= \{\Tstar,g,g^q\}$, and by Corollary~\ref{adult-conic-T}, $\Cplusstar\cap\sistar$ does not meet $\Tstar$. Hence
  $\Cplusstar\cap\Bstar\cap\sistar$
 consists of the four points $\Pbar,\Qbar,\Pqbar,\Qqbar$. Similarly, $\Cplusstarstar\cap\Bstarstar\cap\sistarstar=\{\Pbar,\Qbar, \Pqbar,\Qqbar\}$.
 As $\Cstar\cap\sistar= \Cplusstar\cap\Bstar\cap\sistar$, $\Cstar$ meets $g$ in two distinct points, namely $P,Q$, and so $[\C]$ is a $g$-special normal rational curve. 
 
 In case (ii), {\rm $\Cplus$}  is tangent to $\li$, so by Theorem~\ref{thm:Ccapsi}, $\big\{\Cplusstar\cap\sistar\big\}\cap\big\{\Bstar\cap\sistar\big\}=\{\Pstar\}\ \cap\ \{\Tstar,g,g^q\}=\{\Pbar,\Pqbar\}$. Similarly, $\Cplusstarstar\cap\Bstarstar\cap\sistarstar=\{\Pbar,\Pqbar\}$.
 Hence 
  $\Cstar$ meets $g$ in the repeated point $\Pbar$,  and so $[\C]$ is a $g$-special normal rational curve. 
  
  In case (iii), {\rm $\Cplus$}  is exterior to $\li$, so  in $\PG(2,q^4)$, the extension of $\Cplus$ meets the extension of $\li$ in two points  $\bar P,\, \bar Q$, where $\bar Q=\bar P^{q^2}$.  
 By Theorem~\ref{thm:Ccapsi}, 
  $\Cplusstar\cap\sistar=\emptyset$ and  $\Cplusstarstar\cap\sistarstar=\{\elllP,\elllP^q,\elllP^{q^2},\elllP^{q^3}\}$, where $\elllP=\ellP$. Hence $\Cplusstar\cap\Bstar\cap\sistar=\emptyset$, and $\Cplusstarstar\cap\Bstarstar\cap\sistarstar=\{\Pbar,\Pqbar,\Pbar^{q^2},\Pbar^{q^3}\}$.
So in this case the normal rational curve $[\C]$  meets $\si$ in four points over  ${{\mathbb F}_{q^4}}$. As $\Cstarstar$ meets $\gstar$ in two points (namely $\Pbar$ and  $\Pbar^{q^2}=\Qbar$)  $[\C]$ is an $\gstar$-special normal rational curve.
\end{proof}

We now show that conversely,  every $g$-special or $\gstar$-special normal rational curve corresponds to an $\Fq$-conic.

%

\begin{theorem}\Label{4nrc-is-baby-1}
 Let $\N$ be a $g$-special or $\gstar$-special 4-dimensional normal rational curve in $\PG(4,q)$. Then $\N=[\C]$ where   
$\C$ is an $\Fq$-conic in a tangent Baer subplane of $\PG(2,q^2)$.  
\end{theorem}

\begin{proof} Let $\N$ be a $g$-special 4-dimensional normal rational curve in $\PG(4,q)$. So there are two spread lines $[P]$, $[Q]$ (possibly equal) such that $\Nstar\cap\sistar$ consists of the four points $P=g\cap\Pstar$, $P^q=g^q\cap\Pstar$, $Q=g\cap\Qptstar$, $Q^q=g^q\cap\Qptstar$. Note that as $\Nstar$ meets $\sistar\setminus\si$ in four points, $\N$ is disjoint from $\si$. There are three cases to consider.

Case (i), suppose first that $[P]\neq [Q]$. 
Let $[A],[B],[C]$ be three points  of $\N$, so $[A],[B],[C]\notin\si$. If the plane $\alpha=\langle  [A],[B],[C] \rangle $ contained a point of the spread line  $[P]$, then the $3$-space $\langle \alpha,[P] \rangle\star $ contains five points of $\Nstar$, namely $[A],[B],[C],P,P^q$, a contradiction. So   $\alpha $ is disjoint from the spread lines $[P]$ and $[Q]$.
If $\alpha$ contained a spread line $[X]$, then in $\PG(4,q^2)$,  $\langle \alphastar,g\rangle$ is a 3-space that contains five points of $\Nstar$, namely $[A],[B],[C],P,Q $, a contradiction. So $\alpha$ corresponds to a Baer subplane $\B_\alpha$ of $\PG(2,q^2)$ that is secant to $\li$, and does not contain $\bar P$ or $\bar Q$. 

Consider the corresponding points $\bar P,\, \bar Q,A,B,C$ in $\PG(2,q^2)$. So $\bar P,\, \bar Q\in\li$ and $A,B,C\in\PG(2,q^2)\setminus\li$. Now $A,B,C$ are not collinear as $\alpha$ does not contain a spread line. So  $\B_\alpha$ is the unique Baer subplane that contains $A,B,C$ and is secant to $\li$.  As $\bar P,\, \bar Q\in\li\setminus\B$ and $A,B,C\in\B\setminus\li$,  no three of $\bar P,\, \bar Q,A,B,C$ are collinear, hence they lie on a unique $\Fqq$-conic {\rm $\Cplus$} . By Lemma~\ref{adult-baby}, $A,B,C$ lie in a unique $\Fq$-conic $\C$ contained in {\rm $\Cplus$}, and $\C$ lies in a Baer subplane $\B$. 

Suppose $\B=\B_\alpha$, then by Corollary~\ref{cor:Baerplane-trans}, in $\PG(4,q^2)$, the plane $\alphastar$ meets $\gPQ$. Note that the line $\gPQ$ contains two points  of $\Nstar$, namely  $P,Q^q$. Hence $\langle \alphastar,\gPQ\rangle$ is a 3-space of $\PG(4,q^2)$ that contains five points of $\Nstar$, namely $[A],[B],[C],P, Q^q$, a contradiction. 
Thus $\B\neq\B_\alpha$.

Hence 
 the Baer subplane $\B$ is tangent to $\li$.  As $\Cplus$ is secant to $\li$, we are in case (i) of the proof of Theorem~\ref{smiley-conic}, hence in $\PG(4,q)$, $[\C]$ is a $g$-special 4-dimensional normal rational curve and $\Cstar$ contains the seven points $A,B,C, 
P, P^q, Q, Q^q$. As seven points lie on a unique 4-dimensional normal rational curve, we have  $\Nstar=\Cstar$ and so $\N=[\C]$.
That is, the normal rational curve $\N$ corresponds in $\PG(2,q^2)$  to an $\Fq$-conic $\C$ in the tangent Baer subplane $\B$ as required.

Case (ii), suppose  $[P]= [Q]$, the proof is very similar to case (i). Let $\N$ be a 4-dimensional normal rational curve of $\PG(4,q)$ such that $\N\cap\si=\emptyset$, and $\Nstar\cap\sistar$ consists of two repeated points $P,P^q$.  
Let $[A],[B],[C]\in\N$ and $\alpha=\langle  [A],[B],[C] \rangle$. Similar to case (i), $\alpha$ corresponds to a Baer subplane $\B_\alpha$ of $\PG(2,q^2)$ that is secant to $\li$, and does not contain $\bar P$. 
The points $\bar P,A,B,C$ lie in a unique $\Fqq$-conic $\Cplus$ that is tangent to $\li$ at $\bar P$. By Lemma~\ref{adult-baby}, $A,B,C$ lie in a unique $\Fq$-conic $\C$ contained in {\rm $\Cplus$}, and $\C$ lies in a Baer subplane $\B$. 
If $\B=\B_\alpha$, then
$\bar P\notin\C$, and so $\Cplus$ meets $\li$ in two points, a contradiction. Hence $\B\neq\B_\alpha$ and $\B$ is tangent to $\li$.  As $\Cplus$ is tangent to $\li$, we are in case (ii) of the proof of Theorem~\ref{smiley-conic}, hence in $\PG(4,q)$, $[\C]$ is a $g$-special 4-dimensional normal rational curve  containing $A,B,C$, and 
$\Cstar$ meets $\sistar$ twice at $
P$ and twice at $P^q$. These conditions define a unique normal rational curve of $\PG(4,q^2)$,  and so $\N=\C$ as required.

Case (iii), suppose $\N$ is an $\gstar$-special 4-dimensional normal rational curve. 
As $\N$ is a normal rational curve over ${{\mathbb F}_q}$, $\N$ meets $\sistarstar\setminus\sistar$ in four points which are conjugate with respect to the map $x\mapsto x^q$, $x\in{{\mathbb F}_q}$. That is, points of form $X,X^q,X^{q^2},X^{q^3}$ with $X,X^{q^2}\in \gstar$ and $X^q,X^{q^3}\in \gqstar$.
 Recalling the 1-1 correspondence between points of $\gstar$ and points of the quadratic extension of $\li$ to 
 $\PG(2,q^4)$,  there are points $\bar P,\, \bar Q$ on the  quadratic extension of $\li$ such that $ P =X$, $ Q=X^{q^2}$. 
The argument of case (i)  now generalises by working in 
the quadratic extension of $\PG(2,q^2)$ to $\PG(2,q^4)$; and the quartic extension of $\PG(4,q)$ to $\PG(4,q^4)$. 
\end{proof}

Moreover, the proofs of Theorem~\ref{smiley-conic},~\ref{4nrc-is-baby-1}   show that the normal rational curve corresponding to an $\Fq$-conic $\C$ meets the transversal $g$ of the regular spread $\S$  in points corresponding to the points $\Cplus\cap\li$. The three cases when {\rm $\Cplus$}  is tangent, secant or exterior to $\li$ are summarised in the next result.

\begin{theorem}\Label{baby-not-T-part2}
In $\PG(2,q^2)$, $q>7$, let $\B$ be a Baer subplane tangent to $\li$. Let $\C$ be an $\Fq$-conic in $\B$ with  $\B\cap\li\notin\C$, so $[\C]$ is a 4-dimensional normal rational curve. The $\Fqq$-conic  {\rm $\Cplus$}  meets $\li$ in two  points denoted $\bar P,\, \bar Q$, possibly equal or in an extension. The three possibilities when $\Cplus$ is tangent, secant or exterior to $\li$ are as follows.
\begin{enumerate}
\item $\bar P=\bar Q$ if and only if,  in $\PG(4,q^2)$, $\Cstar$ meets the transversal $g$ of $\S$ in the  point $P$. 
\item $\bar P,\, \bar Q\in\li$ if and only if,  in $\PG(4,q^2)$,  $\Cstar$ meets the transversal $g$ of $\S$ in the two  points $P,Q$. 
\item $\bar P,\, \bar Q$ lie in the extension $\PG(2,q^4)$ if and only if,  in  $\PG(4,q^4)$,   $\Cstarstar$ meets the extended transversal $\gstar$  in the two  points $P$, $Q$. 
\end{enumerate}
\end{theorem}

\bigskip\bigskip

{\bfseries Author information}

S.G. Barwick. School of Mathematical Sciences, University of Adelaide, Adelaide, 5005, Australia.
susan.barwick@adelaide.edu.au

W.-A. Jackson. School of Mathematical Sciences, University of Adelaide, Adelaide, 5005, Australia.
wen.jackson@adelaide.edu.au

P. Wild. Royal Holloway, University of London, TW20 0EX, UK. peterrwild@gmail.com

\end{document}